\documentclass[12pt,british,refpage,intoc,bibliography=totoc,index=totoc,BCOR=7.5mm,captions=tableheading]{extarticle}
\usepackage{mathptmx}
\usepackage[T1]{fontenc}
\usepackage[latin9]{inputenc}
\usepackage[a4paper]{geometry}
\usepackage{color}
\usepackage{babel}
\usepackage{amsmath}
\usepackage{amsthm}
\usepackage{amssymb}
\usepackage{stmaryrd}
\usepackage[numbers]{natbib}
\usepackage[unicode=true,pdfusetitle,
 bookmarks=true,bookmarksnumbered=true,bookmarksopen=false,
 breaklinks=false,pdfborder={0 0 1},backref=false,colorlinks=true]
 {hyperref}
\hypersetup{
 linkcolor=black, citecolor=blue, urlcolor=blue, filecolor=blue, pdfpagelayout=OneColumn, pdfnewwindow=true, pdfstartview=XYZ, plainpages=false}
\usepackage{breakurl}

\makeatletter
\numberwithin{equation}{section}
\numberwithin{figure}{section}
\theoremstyle{plain}
\newtheorem{thm}{\protect\theoremname}[section]
  \theoremstyle{definition}
  \newtheorem{defn}[thm]{\protect\definitionname}
  \theoremstyle{remark}
  \newtheorem{rem}[thm]{\protect\remarkname}
  \theoremstyle{plain}
  \newtheorem{lem}[thm]{\protect\lemmaname}
  \theoremstyle{plain}
  \newtheorem{prop}[thm]{\protect\propositionname}
  \theoremstyle{plain}
  \newtheorem{cor}[thm]{\protect\corollaryname}
  \theoremstyle{definition}
  \newtheorem{example}[thm]{\protect\examplename}

\@ifundefined{date}{}{\date{}}
\usepackage{caption}
\usepackage[nottoc]{tocbibind}

\makeatother

  \providecommand{\corollaryname}{Corollary}
  \providecommand{\definitionname}{Definition}
  \providecommand{\examplename}{Example}
  \providecommand{\lemmaname}{Lemma}
  \providecommand{\propositionname}{Proposition}
  \providecommand{\remarkname}{Remark}
\providecommand{\theoremname}{Theorem}

\begin{document}

\title{Sobolev inequalities on product Sierpinski spaces}

\author{Xuan~Liu\thanks{Mathematical Institute, University of Oxford, Oxford, OX2 6GG, United
Kingdom. Email: \mbox{xuan.liu@maths.ox.ac.uk}.}\hspace{3mm}and Zhongmin~Qian\thanks{Research supported partly by the ERC grant (Grand Agreement No. 291244
ESig). Mathematical Institute, University of Oxford, Oxford OX2 6GG,
United Kingdom. Email: \mbox{zhongmin.qian@maths.ox.ac.uk}.}}
\maketitle
\begin{abstract}
On fractals, different measures (mutually singular in general) are
involved to measure volumes of sets and energies of functions. Singularity
of measures brings difficulties in (especially non-linear) analysis
on fractals. In this paper, we prove a type of Sobolev inequalities,
which involve different and possibly mutually singular measures, on
product Sierpinski spaces. Sufficient and necessary conditions for
the validity of these Sobolev inequalities are given. Furthermore,
we compute the sharp exponents which appears in the sufficient and
necessary conditions for the product Kusuoka measure, \emph{i.e.}
the reference energy measure on Sierpinski spaces.
\end{abstract}
\bigskip{}

\emph{\small{}\ }\textbf{\small{}Keywords.}{\small{}\enskip{}Product
Sierpinski spaces, Sobolev inequalities\medskip{}
}{\small \par}

\textbf{\small{}\ Mathematics Subject Classification}{\small{}.\enskip{}28A80}{\small \par}

\section{\label{sec:-3}Introduction}

Fractals are spaces which in general possess singularities on one
hand, and satisfy certain self-similar properties on the other. A
large amount of research on fractal spaces was motivated by the study
of disordered media in statistical physics. (See \citep{RT83,BH87}.)
Early literatures on analysis on fractals mainly addressed the following
problems:

\medskip{}

(i) constructions of analytic structures (Dirichlet forms in particular)
on fractals and spectral properties of Laplacian operators defined
as the associated self-adjoint operators (see, for example, \citep{Ki89,FS92,Ki93,KL93,BK97,Ki98,MST04}
and references therein);

(ii) constructions of diffusion processes on fractals and heat kernel
estimates for these processes (see, for example, \citep{BP88,BB89,Ham92,Kum93,FHK94,BB99,HK99,Ki08}
and literature quoted in these works).

\medskip{}
Properties of the Laplacian operators on fractals are relatively well
understood. Recent researches on analysis on fractals have been concentrated
on the following topics:

\medskip{}

(iii) function spaces (such as Lipschitz spaces, Sobolev spaces, Besov
spaces and \emph{etc}.) on fractals (\citep{GHL03,Hu03,Str03,HK06,DGV08,LQ17a}
and \emph{etc.});

(iv) gradients on fractals and their applications on (non-linear)
partial differential equations on fractals (see, for example, \citep{Tep00,Hin08,Hin10,HRT13,HT13,HT15,LQ16}
and references listed in these works);

(v) stochastic partial differential equations on fractals (\citep{HRT13,HY16,Yang16}
and \emph{etc.}).

\medskip{}

Most of the existing literature addresses only one-fold fractal spaces,
and research on product spaces of fractals however remains limited
(but see \citep{Str05,BS07} and related references therein). As demonstrated
by Euclidean spaces, analysis on multi-dimensional spaces are more
complicated than that on one-dimensional space. For example, on $\mathbb{R}$
or $[0,1]$, functions with finite energy are automatically Hölder
continuous, however, such functions on $\mathbb{R}^{n}$ or $[0,1]^{n}$
for $n\ge2$ are only Borel measurable. On the other hand, higher
dimensionality introduces more interesting geometric and analytic
features, which allows geometric objects (for example curves and surfaces)
and analytic objects (for example functions and their derivatives)
to have much richer and more profound properties. Similar situation
also happens on fractal spaces. In fact, many fractals share the feature
that the spectral dimension is strictly less than $2$, and this is
crucial to many results on fractals. (Some fractals, depending on
the dimension of the ambient spaces, could have spectral dimensions
larger than $2$. For example, Sierpinski carpets in $\mathbb{R}^{3}$
has spectral dimension greater than $2$; see \citep[p. 706]{BB99}.)

There is a remarkable difference between analysis on Euclidean spaces
and that on fractals: different measures are involved to measure volume
of sets and energy of functions, and these measures are singular to
each other in general. The singularity between these measures introduces
substantial difficulties in non-linear analysis on fractal spaces.
Sobolev inequalities involving singular measures were established
in \citep{LQ17a} for (one-fold) Sierpinski spaces, where these inequalities
were applied to semi-linear partial differential equations on fractals.
In this paper, we consider products of Sierpinski spaces and study
Sobolev inequalities involving singular measures on these spaces.
The main difficulty in establishing these inequalities is that so
far there is no appropriate analogue of the following Newton-Leibniz
formula
\[
f(x)-f(y)=\int_{0}^{1}\big\langle\dot{\gamma}(s),\nabla f(\gamma(s))\big\rangle\;ds,\;\;x,y\in\mathbb{R}^{n},
\]
where $\gamma:[0,1]\to\mathbb{R}^{n}$ is the geodesic (parametrised
by arc length) connecting $x$ and $y$. To overcome this, our main
idea is to exploit the self-similar property of Sierpinski spaces
and derive the Sobolev inequalities by an iteration argument.

The present paper is organized as follows. In Section \ref{sec:},
we set up the notations and briefly review some related results. In
Section \ref{sec:-1}, we formulate and give the proof of Sobolev
inequalities involving singular measures on product Sierpinski spaces.
Sufficient and necessary conditions for any Sobolev inequality to
hold are also given in this section. Section \ref{sec:-2} is devoted
to the sharp values of the exponents in the sufficient and necessary
conditions introduced in Section \ref{sec:-1}. The main difficulty
in computing these sharp values lies in the non-commutativity of the
matrices involved in the harmonic structure on Sierpinski spaces.
Though results in this paper are formulated and proved specifically
for product Sierpinski spaces, however we believe that our methods
should be easily adapted and most of the results (except those in
Section \ref{sec:-2}) should remain valid for products (with possibly
different components) of more general fractals.

\section{\label{sec:}Preliminaries}

In this section, we shall introduce notations that will be in force
throughout this paper, and give a brief review of analysis on Sierpinski
spaces.

\subsubsection*{Sierpinski spaces.}

Let $\mathrm{V}_{0,0}=\{p_{1},p_{2},p_{3}\}\subseteq\mathbb{R}^{2}$
with $p_{1}=(0,0),\;p_{2}=(1,0),\;p_{3}=(1/2,\sqrt{3}/2)$. Let $\mathbf{F}_{i}:\mathbb{R}^{2}\to\mathbb{R}^{2},\;i=1,2,3$
be the contractions defined by 
\[
\mathbf{F}_{i}(x)=2^{-1}(x+p_{i}),\ x\in\mathbb{R}^{2}.
\]
The $m$-lattices $\mathrm{V}_{0,m},\;m\in\mathbb{N}$ are the sets
defined inductively by
\[
\mathrm{V}_{0,m}=\bigcup_{i=1,2,3}\mathbf{F}_{i}(\mathrm{V}_{0,m-1}),\quad m\in\mathbb{N}_{+}.
\]
The one-fold compact Sierpinski space $\mathbb{S}_{0}$ is defined
to be the closure $\mathbb{S}_{0}=\mathrm{closure}\big(\bigcup_{m=0}^{\infty}\mathrm{V}_{0,m}\big)$,
and the one-fold infinite Sierpinski space $\mathbb{S}$ is defined
to be
\[
\mathbb{S}=\bigcup_{m=0}^{\infty}\mathbf{F}_{1}^{-m}\big[\mathbb{S}_{0}\cup(-\mathbb{S}_{0})\big].
\]

Clearly, $\mathbb{S}$ can be written as $\mathbb{S}=\bigcup_{i\in\mathbb{Z}}\mathbb{S}_{i}$,
where $\mathbb{S}_{i},\,i\in\mathbb{Z}$ are non-overlapping translations
of $\mathbb{S}_{0}$. One may order the two-ended sequence $\mathbb{S}_{i},\,i\in\mathbb{Z}$
according to their distances to the origin $p_{1}$. However, the
ordering of $\mathbb{S}_{i},\,i\in\mathbb{Z}$ is not important in
this paper.

We denote by $\mathrm{V}_{i,m}\subseteq\mathbb{S}_{i}$ the $m$-lattice
on $\mathbb{S}_{i}$, which is a translation of the lattice $\mathrm{V}_{0,m}$
on $\mathbb{S}_{0}$, and define the $m$-lattice $\mathrm{V}_{m}$
on $\mathbb{S}$ to be the union
\[
\mathrm{V}_{m}=\bigcup_{i\in\mathbb{Z}}\mathrm{V}_{i,m},\;\;m\in\mathbb{N}.
\]

Regions and points in $\mathbb{S}_{0}$ can be labelled systematically
in the following way. Let
\[
\mathrm{W}_{\ast}=\big\{\omega=\omega_{1}\omega_{2}\omega_{3}\dots\,:\omega_{i}\in\{1,2,3\},i\in\mathbb{N}_{+}\big\}
\]
be the family of infinite sequences $\omega=\omega_{1}\omega_{2}\omega_{3}\dots$
of symbols in $\{1,2,3\}$. For each $\omega\in\mathrm{W}_{\ast}$,
denote by $[\omega]_{m}=\omega_{1}\omega_{2}\dots\omega_{m},\;m\in\mathbb{N}$
the truncation of $\omega$ of length $m$, and define the map
\[
\mathbf{F}_{[\omega]_{m}}=\mathbf{F}_{\omega_{1}\dots\omega_{m}}=\mathbf{F}_{\omega_{1}}\circ\cdots\circ\mathbf{F}_{\omega_{m}}.
\]

\begin{defn}
A subset $S\subseteq\mathbb{S}_{0}$ of the form $\mathbf{F}_{[\omega]_{m}}(\mathbb{S}_{0}),\;\omega\in\mathrm{W}_{\ast},\,m\in\mathbb{N}$
is called a \emph{dyadic simplex} in $\mathbb{S}_{0}$. Similarly,
a subset $S\subseteq\mathbb{S}$ is called a \emph{dyadic simplex}
in $\mathbb{S}$, if $S=\mathbf{F}_{1}^{-k}\circ\mathbf{F}_{[\omega]_{m}}(\mathbb{S}_{0})$
for some $\omega\in\mathrm{W}_{\ast}$ and some $m,k\in\mathbb{N}$.
\end{defn}
For any $\omega\in\mathrm{W}_{\ast}$, since $\mathbf{F}_{i},\,i=1,2,3$
are contractions, the simplexes $\mathbf{F}_{[\omega]_{m}}(\mathbb{S}_{0})$
shrinks to a point $\pi(\omega)$ in $\mathbb{S}_{0}$. In other words,
the set $\bigcap_{m=0}^{\infty}\mathbf{F}_{[\omega]_{m}}(\mathbb{S}_{0})$
contains the point $\pi(\omega)\in\mathbb{S}_{0}$ as its unique element.
The map $\pi:\mathrm{W}_{\ast}\to\mathbb{S}_{0},\;\omega\mapsto\pi(\omega)$
is surjective and gives a systematic labelling to points in $\mathbb{S}_{0}$.
This scheme applies similarly to the translations $\mathbb{S}_{i}$
of $\mathbb{S}_{0}$.
\begin{defn}
The \emph{Hausdorff measure} $\nu$ on $\mathbb{S}$, normalized so
that $\nu(\mathbb{S}_{0})=1$, is the unique Borel measure on $\mathbb{S}$
such that $\nu\big(\mathbf{F}_{[\omega]_{m}}(\mathbb{S}_{i})\big)=3^{-m}$
for all $m\in\mathbb{N},\,i\in\mathbb{Z},\,\omega\in\mathrm{W}_{\ast}$.

\medskip{}
\end{defn}
Notations on product spaces are similar. Specifically, for $n\in\mathbb{N}_{+}$,
the $n$-fold compact Sierpinski space $\mathbb{S}_{0}^{n}$ is defined
to be the $n\text{-}$fold Cartesian product $\mathbb{S}_{0}^{n}=\mathbb{S}_{0}\times\cdots\times\mathbb{S}_{0}$
with its product topology, and the $n$-fold infinite Sierpinski space
$\mathbb{S}^{n}$ is defined to be the $n$-fold Cartesian product
of $\mathbb{S}$. We shall denote a generic point in $\mathbb{S}^{n}$
by $x=(x_{1},\dots,x_{n}),\,x_{i}\in\mathbb{S},\,1\le i\le n$.

For any $n$-tuple $i=(i_{1},\dots,i_{n})\in\mathbb{Z}^{n}$, define
\[
\mathbb{S}_{i}^{n}=\mathbb{S}_{(i_{1},\dots,i_{n})}^{n}=\mathbb{S}_{i_{1}}\times\cdots\times\mathbb{S}_{i_{n}},
\]
where $\mathbb{S}_{j},\,j\in\mathbb{Z}$ are the translations of $\mathbb{S}_{0}$
defined above. Then $\mathbb{S}_{i}^{n},\,i\in\mathbb{Z}^{n}$ are
non-overlapping translations of $\mathbb{S}_{0}^{n}$, and $\mathbb{S}^{n}=\bigcup_{i\in\mathbb{Z}^{n}}\mathbb{S}_{i}^{n}$.
\begin{defn}
For any $n$-tuple $\tau=(\tau_{1},\dots,\tau_{n})\in\{1,2,3\}^{n}$,
we define the map $\mathbf{F}_{\tau}:\mathbb{S}^{n}\to\mathbb{S}^{n}$
by
\[
\begin{aligned}\mathbf{F}_{\tau}(x) & =\mathbf{F}_{(\tau_{1},\dots,\tau_{n})}(x)=\mathbf{F}_{\tau_{1}}\otimes\cdots\otimes\mathbf{F}_{\tau_{n}}(x_{1},\dots,x_{n})\\
 & =\big(\mathbf{F}_{\tau_{1}}(x_{1}),\dots,\mathbf{F}_{\tau_{n}}(x_{n})\big),\quad x=(x_{1},\dots,x_{n})\in\mathbb{S}^{n}.
\end{aligned}
\]
\end{defn}
Note the difference between the notations $\mathbf{F}_{i_{1}\dots i_{n}}$
and $\mathbf{F}_{(i_{1}\dots i_{n})}$. The former denotes the composition
of maps $\mathbf{F}_{i_{k}},\,1\le k\le n$ on $\mathbb{S}$, while
the latter denotes the product of these maps.
\begin{defn}
Let 
\[
\mathrm{W}_{\ast}^{n}=\big\{\omega=\omega_{1}\omega_{2}\omega_{3}\dots\,:\omega_{i}\in\{1,2,3\}^{n},\;i\in\mathbb{N}_{+}\big\}
\]
be the family of infinite sequences $\omega=\omega_{1}\omega_{2}\omega_{3}\dots$
of $n$-tuples $\omega_{i}=(\omega_{i1},\dots,\omega_{in})\in\{1,2,3\}^{n}$.
For each $\omega=\omega_{1}\omega_{2}\dots\,\in\mathrm{W}_{\ast}^{n}$,
similar to the one-fold case, we denote by 
\[
[\omega]_{m}=\omega_{1}\dots\omega_{m}=(\omega_{11},\dots,\omega_{1n})\dots(\omega_{m1},\dots,\omega_{mn})
\]
the truncation of $\omega$ of length $m$, and define the map $\mathbf{F}_{[\omega]_{m}}:\mathbb{S}^{n}\to\mathbb{S}^{n}$
by
\[
\mathbf{F}_{[\omega]_{m}}=\mathbf{F}_{\omega_{1}\dots\omega_{m}}=\mathbf{F}_{(\omega_{11},\dots\omega_{1n})}\circ\cdots\circ\mathbf{F}_{(\omega_{m1},\dots\omega_{mn})}.
\]
\end{defn}
Though the same character $\omega$ is used to denote elements of
$\mathrm{W}_{\ast}$ and $\mathrm{W}_{\ast}^{n}$, it would be clear
from the context that which of the families $\mathrm{W}_{\ast}$ and
$\mathrm{W}_{\ast}^{n}$ is referred to.
\begin{defn}
The \emph{Hausdorff measure} $\nu_{n}$, normalized so that $\nu_{n}(\mathbb{S}_{0}^{n})=1$,
on $\mathbb{S}^{n}$ is defined to be the product $\nu_{n}=\nu\times\cdots\times\nu.$
\end{defn}

\subsubsection*{Standard Dirichlet forms.}

Dirichlet forms on $\mathbb{S}_{0}$ and $\mathbb{S}$ can be introduced
by means of finite difference schemes. (Equivalent definitions of
Dirichlet forms using sequence of random walks are also available.
See, for example, \citep{Gold87,Ku87,BP88} for more details.) For
$m\in\mathbb{N}$ and any function $u$ on the lattice $\mathrm{V}_{0,m}$,
define
\[
\mathcal{E}_{0}^{(m)}(u,u)=\sum_{x,y\in\mathrm{V}_{0,m}:\,|x-y|=2^{-m}}\,\frac{1}{2}\,\Big(\frac{5}{3}\Big)^{m}\;|u(x)-u(y)|^{2}.
\]
The scaling factor $\frac{5}{3}$ is chosen so that the sequence $\{\mathcal{E}_{0}^{(m)}\}$
of forms is consistent; that is, for any function $u$ on $\mathrm{V}_{0,m}$,
\begin{equation}
\mathcal{E}_{0}^{(m)}(u,u)=\min\big\{\mathcal{E}_{0}^{(m+1)}(w,w):w\ \text{is a function on}\ \mathrm{V}_{0,m+1}\ \text{and}\ w\big|_{\mathrm{V}_{0,m}}=u\big\}.\label{eq:-19}
\end{equation}
Clearly, $\mathcal{E}_{0}^{(m+1)}(u,u)=\sum_{i=1,2,3}\,\frac{5}{3}\,\mathcal{E}_{0}^{(m)}(u\circ\mathbf{F}_{i},u\circ\mathbf{F}_{i})$
for all functions $u$ on $\mathrm{V}_{0,m+1}$. For convenience,
we denote
\begin{equation}
\delta_{s}=\frac{1}{2/d_{s}-1},\label{eq:-29}
\end{equation}
where $d_{s}=\frac{2\log3}{\log5}\in(1,2)$ is the spectral dimension
of $\mathbb{S}_{0}$. Then factor $\frac{5}{3}$ can be written as
$\frac{5}{3}=3^{1/\delta_{s}}$.

Let
\[
\mathbf{P}=\left[\begin{array}{ccc}
\vspace{1mm}\frac{2}{3} & -\frac{1}{3} & -\frac{1}{3}\\
\vspace{1mm}-\frac{1}{3} & \frac{2}{3} & -\frac{1}{3}\\
\vspace{1mm}-\frac{1}{3} & -\frac{1}{3} & \frac{2}{3}
\end{array}\right].
\]
Then $\mathcal{E}^{(0)}(u,u)$ can be written as 
\begin{equation}
\mathcal{E}^{(0)}(u,u)=\frac{3}{2}\;u^{\mathrm{t}}\,\mathbf{P}\,u.\label{eq:-27}
\end{equation}

In view of the monotonicity (\ref{eq:-19}) of the sequence $\{\mathcal{E}_{0}^{(m)}\}$,
the limit (possibly infinite)
\[
\mathcal{E}_{0}(u,u)=\lim_{m\to\infty}\mathcal{E}_{0}^{(m)}(u,u)
\]
exists for any function $u$ on $\bigcup_{m=0}^{\infty}\mathrm{V}_{0,m}$.
Moreover, the following self-similar property holds
\begin{equation}
\mathcal{E}_{0}(u,u)=\sum_{i=1,2,3}\frac{5}{3}\,\mathcal{E}_{0}\big(u\circ\mathbf{F}_{i},u\circ\mathbf{F}_{i}\big).\label{eq:-18}
\end{equation}

Let
\[
\mathcal{F}(\mathbb{S}_{0})=\Big\{ u:u\ \text{is a function on}\ \bigcup_{m=0}^{\infty}\mathrm{V}_{0,m}\ \text{and}\ \mathcal{E}_{0}(u,u)<\infty\Big\}.
\]
It is well known that every function $u\in\mathcal{F}(\mathbb{S}_{0})$
is continuous on $\bigcup_{m=0}^{\infty}\mathrm{V}_{0,m}$, and hence
admits a unique continuous extension onto $\mathbb{S}_{0}$. (See,
for example, \citep[Theorem 2.2.6 and Theroem 3.3.4]{Ki01}.) In other
words, we have $\mathcal{F}(\mathbb{S}_{0})\subseteq C(\mathbb{S}_{0})$.
Moreover, the following Poincaré inequality holds
\begin{equation}
\int_{\mathbb{S}_{0}}|u-[u]_{\mathbb{S}_{0}}|^{2}\;d\nu\le C_{\ast}\;\mathcal{E}_{0}(u,u)\;\;\text{for all}\ u\in\mathcal{F}(\mathbb{S}_{0}),\label{eq:-35}
\end{equation}
where $[u]_{\mathbb{S}_{0}}=\int_{\mathbb{S}_{0}}u\,d\nu$, and $C_{\ast}>0$
is a universal constant. (See, for example, \citep[Lemma 2.3.9 and Theorem 3.3.4]{Ki01}
or \citep[Section 2]{LQ17a}.)

The form $(\mathcal{E}_{0},\mathcal{F}(\mathbb{S}_{0}))$, called
the \emph{standard Dirichlet form} on $\mathbb{S}_{0}$, is a local
Dirichlet from on $L^{2}(\mathbb{S}_{0};\nu)$.

\bigskip{}

For any given function $u$ on $\mathrm{V}_{0,0}$, there exists a
unique $h\in\mathcal{F}(\mathbb{S}_{0})$ such that $h|_{\mathrm{V}_{0,0}}=u$
and
\[
\mathcal{E}(h,h)=\min\big\{\mathcal{E}(w,w):w\in\mathcal{F}(\mathbb{S}_{0})\ \text{and}\ w\big|_{\mathrm{V}_{0,0}}=u\big\}.
\]
The function $h\in\mathcal{F}(\mathbb{S}_{0})$ is called the \emph{harmonic
function} in $\mathbb{S}_{0}$ with boundary value $u$, and satisfies
\[
\mathcal{E}(h,h)=\mathcal{E}^{(m)}(h,h)=\mathcal{E}^{(0)}(u,u)\;\;\text{for all}\;m\in\mathbb{N}.
\]
By the above and (\ref{eq:-19}), the value of a harmonic function
$h$ on $\mathrm{V}_{0,1}\backslash\mathrm{V}_{0,0}$ is the extreme
point of a quadratic form, and is given by
\begin{equation}
\big(h\circ\mathbf{F}_{i}\big)\big|_{\mathrm{V}_{0,0}}=\mathbf{A}_{i}\big(h|_{\mathrm{V}_{0,0}}\big),\;\;i=1,2,3,\label{eq:-22}
\end{equation}
where $\mathbf{A}_{i}:\mathbb{R}^{3}\to\mathbb{R}^{3},\,i=1,2,3$
are the linear operators with matrix representations
\begin{equation}
\mathbf{A}_{1}=\left[\begin{array}{ccc}
1 & 0 & 0\\
\vspace{1mm}\frac{2}{5} & \frac{2}{5} & \frac{1}{5}\\
\vspace{1mm}\frac{2}{5} & \frac{1}{5} & \frac{2}{5}
\end{array}\right],\;\mathbf{A}_{2}=\left[\begin{array}{ccc}
\vspace{1mm}\frac{2}{5} & \frac{2}{5} & \frac{1}{5}\\
\vspace{1mm}0 & 1 & 0\\
\vspace{1mm}\frac{1}{5} & \frac{2}{5} & \frac{2}{5}
\end{array}\right],\;\mathbf{A}_{3}=\left[\begin{array}{ccc}
\vspace{1mm}\frac{2}{5} & \frac{1}{5} & \frac{2}{5}\\
\vspace{1mm}\frac{1}{5} & \frac{2}{5} & \frac{2}{5}\\
0 & 0 & 1
\end{array}\right].\label{eq:-21}
\end{equation}

\begin{rem}
The matrices $\mathbf{A}_{i},\,i=1,2,3$ share the same eigenvalues
$\big\{\frac{1}{5},\,\frac{3}{5},\,1\big\}$, and are not mutually
commutative.
\end{rem}
For any $m\in\mathbb{N}$, the value of $h$ on $\mathrm{V}_{0,m}$
can be given by iterations of (\ref{eq:-22})
\begin{equation}
\big(h\circ\mathbf{F}_{[\omega]_{m}}\big)\big|_{\mathrm{V}_{0,0}}=\mathbf{A}_{[\omega]_{m}}\big(h|_{\mathrm{V}_{0,0}}\big),\;\;i=1,2,3\;\;\text{for all}\ \omega\in\mathrm{W}_{\ast},\label{eq:-23}
\end{equation}
where we have used the convention that
\[
\mathbf{A}_{[\omega]_{m}}=\mathbf{A}_{\omega_{1}\omega_{2}\dots\omega_{m}}=\mathbf{A}_{\omega_{m}}\cdots\mathbf{A}_{\omega_{2}}\mathbf{A}_{\omega_{1}}.
\]
Notice that, in the above, the order of subscripts in the product
is reversed. The same notation will be used for all matrices that
appear in this paper.
\begin{defn}
Let $m\in\mathbb{N}$. A function $h\in\mathcal{F}(\mathbb{S}_{0})$
is called an $m$\emph{-harmonic function} in $\mathbb{S}_{0}$, if
$h\circ\mathbf{F}_{[\omega]_{m}}$ is a harmonic function in $\mathbb{S}_{0}$
for all $\omega\in\mathrm{W}_{\ast}$. A function $h$ is called \emph{piecewise
harmonic} if it is $m$-harmonic for some $m\in\mathbb{N}$.
\end{defn}
\medskip{}

The standard Dirichlet form on $\mathbb{S}$ can be defined similarly.
For any function $u$ on $\bigcup_{m=0}^{\infty}\mathrm{V}_{m}$,
define
\[
\mathcal{E}^{(m)}(u,u)=\sum_{x,y\in\mathrm{V}_{m}:\,|x-y|=2^{-m}}\,\frac{1}{2}\,\Big(\frac{5}{3}\Big)^{m}\;|u(x)-u(y)|^{2},
\]
\[
\mathcal{E}(u,u)=\lim_{m\to\infty}\mathcal{E}^{(m)}(u,u).
\]
The form $\mathcal{E}$ satisfies the following self-similar property
\begin{equation}
\mathcal{E}(u,u)=\frac{5}{3}\,\mathcal{E}\big(u\circ\mathbf{F}_{1}\big).\label{eq:-30}
\end{equation}

Every function $u$ on $\bigcup_{m=0}^{\infty}\mathrm{V}_{m}$ with
$\mathcal{E}(u,u)<\infty$ admits a unique continuous extension onto
$\mathbb{S}$. Let
\[
\mathcal{F}(\mathbb{S})=L^{2}(\mathbb{S};\nu)\cap\Big\{ u:u\ \text{is a function on}\ \bigcup_{m=0}^{\infty}\mathrm{V}_{m}\ \text{and}\ \mathcal{E}(u,u)<\infty\Big\}.
\]
Then $\mathcal{F}(\mathbb{S})\subseteq L^{2}(\mathbb{S};\nu)\cap C_{0}(\mathbb{S})$,
where $C_{0}(\mathbb{S})$ is the space of continuous functions on
$\mathbb{S}$ vanishing at infinity.

The form $(\mathcal{E},\mathcal{F}(\mathbb{S}))$, called the \emph{standard
Dirichlet form }on $\mathbb{S}$, is a local Dirichlet form on $L^{2}(\mathbb{S};\nu)$.

\subsubsection*{Kusuoka measure and gradients.}
\begin{defn}
For any $u\in\mathcal{F}(\mathbb{S}_{0})$, the \emph{energy measure}
$\mu_{\langle u\rangle}$ of $u$ is the unique Borel measure on $\mathbb{S}_{0}$
such that
\[
\int_{\mathbb{S}_{0}}\phi\,d\mu_{\langle u\rangle}=2\mathcal{E}_{0}(\phi u,u)-\mathcal{E}_{0}(\phi,u^{2})\;\;\text{for all}\ \phi\in\mathcal{F}(\mathbb{S}_{0}).
\]
For any $u,w\in\mathcal{F}(\mathbb{S}_{0})$, the \emph{mutual energy
measure} $\mu_{\langle u,w\rangle}$ is defined by the polarisation
$\mu_{\langle u,w\rangle}=\frac{1}{4}(\mu_{\langle u+w\rangle}-\mu_{\langle u-w\rangle})$.
\end{defn}
By definition, $\mu_{\langle u\rangle}(\mathbb{S})=\mathcal{E}_{0}(u,u)$,
which, together with the self-similar property (\ref{eq:-18}), implies
that
\begin{equation}
\mu_{\langle u\rangle}\big(\mathbf{F}_{[\omega]_{m}}(\mathbb{S}_{0})\big)=\Big(\frac{5}{3}\Big)^{m}\,\mathcal{E}_{0}\big(u\circ\mathbf{F}_{[\omega]_{m}},u\circ\mathbf{F}_{[\omega]_{m}}\big)\;\;\text{for all}\ \omega\in\mathrm{W}_{\ast},\,m\in\mathbb{N}.\label{eq:-25}
\end{equation}
In particular, for a harmonic function $h$ with boundary value $h|_{\mathrm{V}_{0,0}}=u$,
by (\ref{eq:-27}) and (\ref{eq:-23}),
\begin{equation}
\mu_{\langle h\rangle}\big(\mathbf{F}_{[\omega]_{m}}(\mathbb{S}_{0})\big)=\frac{3}{2}\,\Big(\frac{5}{3}\Big)^{m}\;u^{\mathrm{t}}\,\mathbf{Y}_{[\omega]_{m}}^{\mathrm{t}}\mathbf{Y}_{[\omega]_{m}}u,\label{eq:-26}
\end{equation}
where $\mathbf{Y}_{i}=\mathbf{P}^{\mathrm{t}}\,\mathbf{A}_{i}\mathbf{P},\,i=1,2,3$,
and we used the fact $\mathbf{P}\,\mathbf{A}_{i}=\mathbf{P}^{\mathrm{t}}\,\mathbf{A}_{i}\mathbf{P}$.

For later use, we write down these matrices explicitly
\begin{equation}
\mathbf{Y}_{1}=\left[\begin{array}{ccc}
\vspace{1mm}\frac{2}{5} & -\frac{1}{5} & -\frac{1}{5}\\
\vspace{1mm}-\frac{1}{5} & \frac{1}{5} & 0\\
\vspace{1mm}-\frac{1}{5} & 0 & \frac{1}{5}
\end{array}\right],\,\mathbf{Y}_{2}=\left[\begin{array}{ccc}
\vspace{1mm}\frac{1}{5} & -\frac{1}{5} & 0\\
\vspace{1mm}-\frac{1}{5} & \frac{2}{5} & -\frac{1}{5}\\
\vspace{1mm}0 & -\frac{1}{5} & \frac{1}{5}
\end{array}\right],\,\mathbf{Y}_{3}=\left[\begin{array}{ccc}
\vspace{1mm}\frac{1}{5} & 0 & -\frac{1}{5}\\
\vspace{1mm}0 & \frac{1}{5} & -\frac{1}{5}\\
\vspace{1mm}-\frac{1}{5} & -\frac{1}{5} & \frac{2}{5}
\end{array}\right].\label{eq:-24}
\end{equation}

\begin{rem}
The matrices $\mathbf{Y}_{i},\,i=1,2,3$ share the same eigenvalues
$\big\{0,\,\frac{1}{5},\,\frac{3}{5}\big\}$, and are not mutually
commutative.
\end{rem}
\begin{defn}
The \emph{Kusuoka measure} $\mu$ on $\mathbb{S}$ is the unique Borel
measure on $\mathbb{S}$ such that
\begin{equation}
\mu\big(\mathbf{F}_{[\omega]_{m}}(\mathbb{S}_{i})\big)=\frac{1}{2}\,\Big(\frac{5}{3}\Big)^{m}\;\mathrm{trace}\big(\mathbf{Y}_{[\omega]_{m}}^{\mathrm{t}}\mathbf{Y}_{[\omega]_{m}}\big)\ \;\text{for all}\ m\in\mathbb{N},\,i\in\mathbb{Z},\,\omega\in\mathrm{W}_{\ast}.\label{eq:-20}
\end{equation}
\end{defn}
\begin{rem}
(i) The factor $\frac{1}{2}$ in (\ref{eq:-20}) normalizes $\mu$
so that $\mu(\mathbb{S}_{0})=1$.

(ii) Let $h_{i},\,i=1,2,3$ be the harmonic functions with boundary
values $h_{i}|_{\mathrm{V}_{0,0}}=1_{\{p_{i}\}}$. It is easily seen
from (\ref{eq:-26}) and the definition of $\mu$ that
\begin{equation}
\mu=\frac{1}{3}\,\big(\mu_{\langle h_{1}\rangle}+\mu_{\langle h_{2}\rangle}+\mu_{\langle h_{3}\rangle}\big).\label{eq:-28}
\end{equation}

(iii) The Kusuoka measure $\mu$ and the Hausdorff measure $\nu$
are mutually singular. (See \citep[Example 1, Section 6]{Ku89}.)
\end{rem}
By (\ref{eq:-28}), for any harmonic function $h$, its energy measure
$\mu_{\langle h\rangle}$ is absolutely continuous with respect to
$\mu$. The same is true for any piecewise harmonic function in view
of the self-similar property (\ref{eq:-18}). By the denseness of
piecewise harmonic functions in $\mathcal{F}(\mathbb{S}_{0})$, we
see that $\mu_{\langle u\rangle}\ll\mu$ for all $u\in\mathcal{F}(\mathbb{S}_{0})$.

Let $h_{1}$ be the harmonic function in $\mathbb{S}_{0}$ with boundary
value $h_{1}\big|_{\mathrm{V}_{0,0}}=1_{\{p_{1}\}}$, and define
\begin{equation}
\nabla h_{1}=-\sqrt{\frac{d\mu_{\langle h_{1}\rangle}}{d\mu}},\label{eq:-32}
\end{equation}
where the minus sign appears in the above definition as a convention.
Note that $\nabla h_{1}<0\;\;\mu$-a.e. as $\mu_{\langle h_{1}\rangle}(\mathbf{F}_{[\omega]_{m}}(\mathbb{S}_{0}))>0$
for all $\omega\in\mathrm{W}_{\ast},\,m\in\mathbb{N}$. 
\begin{defn}
For any $u\in\mathcal{F}(\mathbb{S}_{0})$, the \emph{gradient} $\nabla u$
of $u$ is defined to be
\begin{equation}
\nabla u=(\nabla h_{1})^{-1}\;\frac{d\mu_{\langle u,h_{1}\rangle}}{d\mu}.\label{eq:-31}
\end{equation}
\end{defn}
\begin{defn}
For any $u\in\mathcal{F}(\mathbb{S})$, the \emph{gradient} $\nabla u$
of $u$ is defined to be
\[
\nabla u=\sum_{i\in\mathbb{Z}}\nabla\big(u|_{\mathbb{S}_{i}}\big),
\]
where the gradients $\nabla\big(u|_{\mathbb{S}_{i}}\big)$ are taken
with $u|_{\mathbb{S}_{i}}$ regarded as functions in $\mathcal{F}(\mathbb{S}_{0})$.
\end{defn}
By the definition of gradients, we have the representations
\[
\mathcal{E}_{0}(u,u)=\int_{\mathbb{S}_{0}}|\nabla u|^{2}\;d\mu,\;\;\mathcal{E}(u,u)=\int_{\mathbb{S}}|\nabla u|^{2}\;d\mu.
\]
Moreover, for $G\in C^{1}(\mathbb{R}^{k})$ and $u_{1},\dots,u_{k}\in\mathcal{F}(\mathbb{S}_{0})$,
the following chain rule holds
\[
\nabla G(u_{1},\dots,u_{k})=\sum_{i=1}^{k}\partial_{i}G(u_{1},\dots,u_{k})\nabla u_{i}.
\]

\section{\label{sec:-1}Sobolev inequalities}

We start with two lemmas regarding elementary properties of gradients
and harmonic functions in $\mathbb{S}_{0}$ respectively.
\begin{lem}
\label{lem:}Let $u\in\mathcal{F}(\mathbb{S}_{0})$. Then for any
$\omega\in\mathrm{W}_{\ast}$ and any $m\in\mathbb{N}_{+}$,

(a)
\begin{equation}
\mathcal{E}\big(u\circ\mathbf{F}_{[\omega]_{m}}\big)=\Big(\frac{3}{5}\Big)^{m}\int_{\mathbf{F}_{[\omega]_{m}}(\mathbb{S}_{0})}|\nabla u|^{2}\,d\mu.\label{eq:-17}
\end{equation}

(b) 
\begin{equation}
\nabla\big(u\circ\mathbf{F}_{[\omega]_{m}}\big)=\Big(\frac{3}{5}\Big)^{m/2}\;\big(\nabla u\circ\mathbf{F}_{[\omega]_{m}}\big)\cdot\bigg[\frac{d(\mu\circ\mathbf{F}_{[\omega]_{m}})}{d\mu}\bigg]^{1/2}.\label{eq:-2}
\end{equation}

(c) For $\omega\in\mathrm{W}_{\ast}$ and $m\in\mathbb{N}$,
\begin{equation}
\Big(\frac{1}{15}\Big)^{m}\le\frac{d(\mu\circ\mathbf{F}_{[\omega]_{m}})}{d\mu}\le\Big(\frac{3}{5}\Big)^{m},\quad\mu\text{-a.e.}\label{eq:-3}
\end{equation}
Consequently, for any $r\ge2$,
\begin{equation}
\begin{aligned}\bigg(\frac{3^{r^{\prime}/r}}{5}\bigg)^{m/r^{\prime}}\;\big|\nabla & u\circ\mathbf{F}_{[\omega]_{m}}\big|\cdot\bigg[\frac{d(\mu\circ\mathbf{F}_{[\omega]_{m}})}{d\mu}\bigg]^{1/r}\\
 & \le\big|\nabla\big(u\circ\mathbf{F}_{[\omega]_{m}}\big)\big|\le\Big(\frac{3}{5}\Big)^{m/r^{\prime}}\;\big|\nabla u\circ\mathbf{F}_{[\omega]_{m}}\big|\cdot\bigg[\frac{d(\mu\circ\mathbf{F}_{[\omega]_{m}})}{d\mu}\bigg]^{1/r}.
\end{aligned}
\label{eq:-4}
\end{equation}
\end{lem}
\begin{proof}
(a) This is an immediate corollary of (\ref{eq:-25}) and the definition
of gradients.

(b) For any $\omega,\omega^{\prime}\in\mathrm{W}_{\ast}$ and any
$m,l\in\mathbb{N}_{+}$, by (a),
\[
\begin{aligned}\int_{\mathbf{F}_{[\omega^{\prime}]_{l}}(\mathbb{S}_{0})}\big|\nabla\big(u\circ\mathbf{F}_{[\omega]_{m}}\big)\big|^{2}\,d\mu & =\Big(\frac{5}{3}\Big)^{l}\;\mathcal{E}\big(u\circ\mathbf{F}_{[\omega]_{m}}\circ\mathbf{F}_{[\omega^{\prime}]_{m}}\big)\\
 & =\Big(\frac{3}{5}\Big)^{m}\int_{\mathbf{F}_{[\omega]_{m}}\circ\mathbf{F}_{[\omega^{\prime}]_{l}}(\mathbb{S}_{0})}|\nabla u|^{2}\,d\mu\\
 & =\Big(\frac{3}{5}\Big)^{m}\int_{\mathbf{F}_{[\omega^{\prime}]_{l}}(\mathbb{S}_{0})}\big|\nabla u\circ\mathbf{F}_{[\omega]_{m}}\big|^{2}\,d\big(\mu\circ\mathbf{F}_{[\omega]_{m}}\big),
\end{aligned}
\]
which implies that
\begin{equation}
\big|\nabla\big(u\circ\mathbf{F}_{[\omega]_{m}}\big)\big|^{2}=\Big(\frac{3}{5}\Big)^{m}\;\big|\nabla u\circ\mathbf{F}_{[\omega]_{m}}\big|^{2}\cdot\frac{d(\mu\circ\mathbf{F}_{[\omega]_{m}})}{d\mu}.\label{eq:-33}
\end{equation}
In particular,
\begin{equation}
\big|\nabla\big(h_{1}\circ\mathbf{F}_{i}\big)\big|=\sqrt{\frac{3}{5}}\;\big|\nabla h_{1}\circ\mathbf{F}_{i}\big|\cdot\frac{d(\mu\circ\mathbf{F}_{i})}{d\mu},\;i=1,2,3,\label{eq:-34}
\end{equation}
where $h_{1}$ is the harmonic function with boundary value $h_{1}|_{\mathrm{V}_{0,0}}=1_{\{p_{1}\}}$.
Note that, by (\ref{eq:-22}), $h_{1}\circ\mathbf{F}_{1}=\frac{2}{5}+\frac{3}{5}\,h_{1}$
and $h_{1}\circ\mathbf{F}_{2}=h_{1}\circ\mathbf{F}_{3}=\frac{2}{5}\,h_{1}$.
Therefore, 
\[
\nabla(h_{1}\circ\mathbf{F}_{1})=\frac{3}{5}\nabla h_{1},\;\nabla(h_{1}\circ\mathbf{F}_{2})=\nabla(h_{1}\circ\mathbf{F}_{2})=\frac{2}{5}\nabla h_{1}.
\]
In view of (\ref{eq:-32}), we see that $\nabla(h_{1}\circ\mathbf{F}_{i})<0\;\;\mu\text{-a.e.},\,i=1,2,3$.
Since $\nabla h_{1}<0\;\;\mu$-a.e., it follows from (\ref{eq:-34})
that
\[
\nabla\big(h_{1}\circ\mathbf{F}_{i}\big)=\sqrt{\frac{3}{5}}\;\nabla h_{1}\circ\mathbf{F}_{i}\cdot\frac{d(\mu\circ\mathbf{F}_{i})}{d\mu},\;i=1,2,3.
\]
By the above and induction, it is easily seen that (\ref{eq:-2})
holds for $h_{1}$. Moreover, (\ref{eq:-2}) for general $u\in\mathcal{F}(\mathbb{S}_{0})$
follows from (\ref{eq:-2}) for $h_{1}$ and polarisation of (\ref{eq:-33}).

(c) We only need to prove (\ref{eq:-3}). The inequality (\ref{eq:-4})
follows from (\ref{eq:-3}) immediately. For any $\omega^{\prime}\in\mathrm{W}_{\ast}$
and any $l\in\mathbb{N}$, since $\text{spectrum}\big(\mathbf{Y}_{i}\big)=\big\{0,\,\frac{1}{5},\,\frac{3}{5}\big\}$,
we have
\[
\begin{aligned}\int_{\mathbf{F}_{[\omega^{\prime}]_{l}}(\mathbb{S}_{0})}d\big(\mu\circ\mathbf{F}_{[\omega]_{m}}\big) & =\mu\big(\mathbf{F}_{[\omega]_{m}}\circ\mathbf{F}_{[\omega^{\prime}]_{l}}(\mathbb{S}_{0})\big)=\Big(\frac{5}{3}\Big)^{m+l}\;\text{trace}\big(\mathbf{Y}_{[\omega^{\prime}]_{l}}^{\mathrm{t}}\mathbf{Y}_{[\omega]_{m}}^{\mathrm{t}}\mathbf{Y}_{[\omega]_{m}}\mathbf{Y}_{[\omega^{\prime}]_{l}}\big)\\
 & \le\Big(\frac{5}{3}\Big)^{m+l}\Big(\frac{3}{5}\Big)^{2m}\;\text{trace}\big(\mathbf{Y}_{[\omega^{\prime}]_{l}}^{\mathrm{t}}\mathbf{Y}_{[\omega^{\prime}]_{l}}\big)=\Big(\frac{3}{5}\Big)^{m}\;\mu\big(\mathbf{F}_{[\omega^{\prime}]_{l}}(\mathbb{S}_{0})\big).
\end{aligned}
\]
and similarly,
\[
\int_{\mathbf{F}_{[\omega^{\prime}]_{l}}(\mathbb{S}_{0})}d\big(\mu\circ\mathbf{F}_{[\omega]_{m}}\big)\ge\Big(\frac{5}{3}\Big)^{m+l}\Big(\frac{1}{5}\Big)^{2m}\;\text{trace}\big(\mathbf{Y}_{[\omega^{\prime}]_{l}}^{\mathrm{t}}\mathbf{Y}_{[\omega^{\prime}]_{l}}\big)=\Big(\frac{1}{15}\Big)^{m}\;\mu\big(\mathbf{F}_{[\omega^{\prime}]_{l}}(\mathbb{S}_{0})\big).
\]
Now (\ref{eq:-3}) follows readily from the above and the Lebesgue
differentiation theorem.
\end{proof}
\begin{lem}
\label{lem:-5}Let $h_{i},\,i=1,2,3$ be the harmonic functions in
$\mathbb{S}_{0}$ with boundary values $h_{i}\big|_{\mathrm{V}_{0,0}}=1_{\{p_{i}\}},\,i=1,2,3$.
Then

(a) $h_{1}+h_{2}+h_{3}=1$, $|\nabla h_{1}|^{2}+|\nabla h_{2}|^{2}+|\nabla h_{3}|^{2}=3\;\;\mu$-a.e.,
and

\begin{equation}
|\nabla h_{i}|\le\sqrt{2}\ \;\text{on}\ \mathbb{S}_{0}\;\;\mu\text{-a.e.},\,i=1,2,3.\label{eq:-38}
\end{equation}

(b) $|\nabla h_{i}|$ has no strictly positive lower bound in any
dyadic simplex $S=\mathbf{F}_{[\omega]_{m}}(\mathbb{S}_{0})\subseteq\mathbb{S}_{0}$:
\[
\mathop{\mathrm{ess}\,\mathrm{inf}}_{S}\,|\nabla h_{i}|=0,\;\;i=1,2,3,
\]
where the essential infimum is taken with respect to the Kusuoka measure
$\mu$.
\end{lem}
\begin{proof}
(a) The two identities in the statement are corollaries of the uniqueness
of harmonic functions and the fact $\mu=\frac{1}{3}(\mu_{\langle h_{1}\rangle}+\mu_{\langle h_{2}\rangle}+\mu_{\langle h_{3}\rangle})$.

Since $h_{1}+h_{2}+h_{3}=1$, we have $\sum_{i=1,2,3}\nabla h_{i}=0$,
which together with $\sum_{i=1,2,3}|\nabla h_{i}|^{2}=3$ gives
\[
3=\sum_{i=1,2,3}|\nabla h_{i}|^{2}=2(|\nabla h_{1}|^{2}+\nabla h_{1}\nabla h_{2}+|\nabla h_{2}|^{2})\ge\frac{3}{2}|\nabla h_{1}|^{2}.
\]
Therefore, $|\nabla h_{i}|\le\sqrt{2}\;\;\mu$-a.e., $i=1,2,3$.

(b) It suffices to prove $\mathop{\mathrm{ess}\,\mathrm{inf}}_{S}\,|\nabla h_{1}|=0$.
We first show that $\mathop{\mathrm{ess}\,\mathrm{inf}}_{\mathbb{S}_{0}}\,|\nabla h_{1}|=0$.
Let $\omega=2333\dots\,\in\mathrm{W}_{\ast}$. Then
\[
\mathbf{Y}_{[\omega]_{m+1}}=\mathbf{Y}_{3}^{m}\,\mathbf{Y}_{2}=\left[\begin{array}{ccc}
\vspace{2mm}\big(\frac{1}{5}\big)^{m+1} & -\frac{1}{30}\big(\frac{3}{5}\big)^{m}-\frac{3}{10}\big(\frac{1}{5}\big)^{m} & -\frac{1}{30}\big(\frac{3}{5}\big)^{m}+\frac{1}{10}\big(\frac{1}{5}\big)^{m}\\
\vspace{2mm}-\big(\frac{1}{5}\big)^{m+1} & -\frac{1}{30}\big(\frac{3}{5}\big)^{m}+\frac{3}{10}\big(\frac{1}{5}\big)^{m} & -\frac{1}{30}\big(\frac{3}{5}\big)^{m}-\frac{1}{10}\big(\frac{1}{5}\big)^{m}\\
\vspace{2mm}0 & \frac{1}{15}\big(\frac{3}{5}\big)^{m} & \frac{1}{15}\big(\frac{3}{5}\big)^{m}
\end{array}\right].
\]
Therefore,
\[
\frac{1}{\mu(\mathbf{F}_{[\omega]_{m+1}}(\mathbb{S}_{0}))}\int_{\mathbf{F}_{[\omega]_{m+1}}}|\nabla h_{1}|^{2}\;d\mu=\frac{3}{2}\,\frac{\mathbf{e}_{1}^{\mathrm{t}}\mathbf{Y}_{[\omega]_{m+1}}^{\mathrm{t}}\mathbf{Y}_{[\omega]_{m+1}}\mathbf{e}_{1}}{\mathrm{trace}(\mathbf{Y}_{[\omega]_{m+1}}^{\mathrm{t}}\mathbf{Y}_{[\omega]_{m+1}})}\le9^{-m+1}.
\]
The above implies that
\[
\mu\big\{|\nabla h_{1}|\le3^{-m+1}\big\}>0.
\]
Therefore, $\mathop{\mathrm{ess}\,\mathrm{inf}}_{\mathbb{S}_{0}}\,|\nabla h_{1}|=0$.

Next, we show that $\mathop{\mathrm{ess}\,\mathrm{inf}}_{\mathbf{F}_{i}(\mathbb{S}_{0})}\,|\nabla h_{1}|=0,\,i=1,2,3$.
As seen in the proof of Lemma \ref{lem:}-(b), we have
\[
\nabla(h_{1}\circ\mathbf{F}_{1})=\frac{3}{5}\nabla h_{1},\;\nabla(h_{1}\circ\mathbf{F}_{2})=\nabla(h_{1}\circ\mathbf{F}_{2})=\frac{2}{5}\nabla h_{1}.
\]
This implies $\mathop{\mathrm{ess}\,\mathrm{inf}}_{\mathbb{S}_{0}}\,|\nabla(h_{1}\circ\mathbf{F}_{i})|=0$.
Since $\mu\circ\mathbf{F}_{i}$ and $\mu$ are equivalent measures
by virtue of (\ref{eq:-3}), it follows from (\ref{eq:-2}) that
\[
\mathop{\mathrm{ess}\,\mathrm{inf}}_{\mathbf{F}_{i}(\mathbb{S}_{0})}\,|\nabla h_{1}|=\mathop{\mathrm{ess}\,\mathrm{inf}}_{\mathbb{S}_{0}}\,|\nabla h_{1}\circ\mathbf{F}_{i}|=0.
\]

By induction, we see that $\mathop{\mathrm{ess}\,\mathrm{inf}}_{S}\,|\nabla h_{1}|=0$
holds for all dyadic simplexes. This completes the proof of (b).
\end{proof}
We now give the definition of Sobolev spaces $W^{1,r}$ on $\mathbb{S}_{0}^{n}$
and $\mathbb{S}^{n}$.
\begin{defn}
\label{def:}Let $r\ge2$ and $u\in\mathcal{F}(\mathbb{S}_{0})$.
Define
\[
\llbracket u\rrbracket_{W^{1,r}(\mathbb{S}_{0})}=\Vert\nabla u\Vert_{L^{r}(\mathbb{S}_{0};\mu)},
\]
\[
\Vert u\Vert_{W^{1,r}(\mathbb{S}_{0})}=\Big(\Vert u\Vert_{L^{r}(\mathbb{S}_{0};\nu)}^{r}+\llbracket u\rrbracket_{W^{1,r}(\mathbb{S}_{0})}^{r}\Big)^{1/r}.
\]
The \emph{Sobolev space} $W^{1,r}(\mathbb{S}_{0})$ is defined to
be the completion of 
\[
\big\{ u\in\mathcal{F}(\mathbb{S}_{0}):\Vert u\Vert_{W^{1,r}(\mathbb{S}_{0})}<\infty\big\}
\]
with respect to the norm $\Vert\cdot\Vert_{W^{1,r}(\mathbb{S}_{0})}$.
\end{defn}
For a generic point $x=(x_{1},\dots,x_{i},\dots,x_{n})\in\mathbb{S}_{0}^{n}$,
we denote $\widehat{x_{i}}=(x_{1},\dots,x_{i-1},x_{i+1},\dots,x_{n})$.
To simplify notations, we abuse notations and denote
\[
(x_{i},\widehat{x_{i}})=(x_{1},\dots,x_{i},\dots,x_{n}),
\]
\[
(\mu\times\nu_{n-1})(dx_{i},d\widehat{x_{i}})=(\nu\times\cdots\times\mu\times\cdots\times\nu)(dx_{1},\dots,dx_{i},\dots,dx_{n}).
\]

\begin{defn}
Let $r\ge2$, and let $\mathcal{C}(\mathbb{S}_{0}^{n})$ be the space
of functions $u\in C(\mathbb{S}_{0}^{n})$ such that $u(\cdot,\widehat{x_{i}})\in\mathcal{F}(\mathbb{S}_{0})$
for all $1\le i\le n$ and $\widehat{x_{i}}=(x_{1},\dots x_{i-1},x_{i+1},\dots,x_{n})$.
For any $u\in\mathcal{C}(\mathbb{S}_{0}^{n})$, define 
\[
\llbracket u\rrbracket_{W^{1,r}(\mathbb{S}_{0}^{n})}=\Big(\sum_{i=1}^{n}\int_{\mathbb{S}_{0}^{n-1}}|\nabla_{i}\,u(x_{i},\widehat{x_{i}})|^{r}\;(\mu\times\nu_{n-1})(dx_{i},d\widehat{x_{i}})\Big)^{1/r},
\]
\[
\Vert u\Vert_{W^{1,r}(\mathbb{S}_{0}^{n})}=\Big(\Vert u\Vert_{L^{r}(\mathbb{S}_{0}^{n};\nu_{n})}^{r}+\llbracket u\rrbracket_{W^{1,r}(\mathbb{S}_{0}^{n})}^{r}\Big)^{1/r},
\]
where $\nabla_{i}$ refers to the operator $\nabla$ applied to the
$i$-th variable $x_{i}$. The \emph{Sobolev space} $W^{1,r}(\mathbb{S}_{0}^{n})$
is defined to be the completion of 
\[
\big\{ u\in\mathcal{C}(\mathbb{S}_{0}^{n}):\Vert u\Vert_{W^{1,r}(\mathbb{S}_{0}^{n})}<\infty\big\}
\]
with respect to the norm $\Vert\cdot\Vert_{W^{1,r}(\mathbb{S}_{0}^{n})}$.
\end{defn}
The proposition below states that the space $W^{1,r}(\mathbb{S}_{0})$
is sufficiently large.
\begin{prop}
\label{prop:-2}Let $n\ge1,\,r\ge2$. Then the space $W^{1,r}(\mathbb{S}_{0})$
contains all piecewise harmonic functions in $\mathbb{S}_{0}$, and
$C(\mathbb{S}_{0}^{n})\cap W^{1,r}(\mathbb{S}_{0}^{n})$ is dense
in the space $C(\mathbb{S}_{0}^{n})$ with respect to the supremum
norm.
\end{prop}
\begin{proof}
We first show that $W^{1,r}(\mathbb{S}_{0})$ contains all harmonic
functions in $\mathbb{S}_{0}$. Let $h_{i},\,i=1,2,3$ be the harmonic
functions with boundary values 
\[
h_{i}\big|_{\mathrm{V}_{0,0}}=1_{\{p_{i}\}},\,i=1,2,3.
\]
By Lemma \ref{lem:-5}, $\nabla h_{i}\in L^{\infty}(\mathbb{S}_{0};\mu)$
and therefore $h_{i}\in W^{1,r}(\mathbb{S}_{0}),\,i=1,2,3$. Furthermore,
$W^{1,r}(\mathbb{S}_{0})$ contains all harmonic functions in $\mathbb{S}_{0}$
as any harmonic function is a linear combination of $h_{i},\,i=1,2,3$.

By (\ref{eq:-4}) and the above, we see that $W^{1,r}(\mathbb{S}_{0})$
contains all piecewise harmonic functions in $\mathbb{S}_{0}$.

Since the linear space generated by functions of the form
\[
u(x)=u_{1}(x_{1})\cdots u_{n}(x_{n}),\quad u_{1},\dots,u_{n}\in C(\mathbb{S}_{0})
\]
is dense in $C(\mathbb{S}_{0}^{n})$, so is the linear space generated
by functions of the above form with each $u_{i}$ being piecewise
harmonic. Clearly, $u_{1}(x_{1})\cdots u_{n}(x_{n})\in W^{1,r}(\mathbb{S}_{0}^{n})$
when $u_{i},\,i=1,\dots,n$ are piecewise harmonic. This completes
the proof.
\end{proof}
The following Poincaré inequality on $\mathbb{S}^{n}$, which is available
on most fractal spaces, is the cornerstone of our arguments.
\begin{lem}[Poincaré inequality]
There exists a universal constant $C_{\ast}>0$ (also independent
of $n$) such that
\begin{equation}
\int_{\mathbb{S}_{0}^{n}}\big|u-[u]_{\mathbb{S}_{0}^{n}}\big|^{2}\,d\nu_{n}\le C_{\ast}\,\llbracket u\rrbracket_{W^{1,2}(\mathbb{S}_{0}^{n})}^{2},\;u\in W^{1,2}(\mathbb{S}_{0}^{n}),\label{eq:}
\end{equation}
where $[u]_{\mathbb{S}_{0}^{n}}=\int_{\mathbb{S}_{0}^{n}}u\;d\nu_{n}$.
\end{lem}
\begin{proof}
Let $u_{k}$ be the function on $\mathbb{S}_{0}^{k}$ defined by
\[
u_{k}(x_{1},\dots,x_{k})=\int_{\mathbb{S}_{0}^{n-k}}u(x_{1},\dots,x_{n})\;\nu_{n-k}(dx_{k+1},\dots dx_{n}),\;\;k=0,1,\dots,n.
\]
In particular, $u_{0}=[u]_{\mathbb{S}_{0}^{n}}$ and $u_{n}=u$. Clearly,
\[
u_{k-1}(x_{1},\dots,x_{k-1})=\int_{\mathbb{S}_{0}}u_{k}(x_{1},\dots,x_{k})\;\nu(dx_{k}),\;\;k=1,\dots,n.
\]
Moreover,
\begin{equation}
\int_{\mathbb{S}_{0}^{n}}\big|u-[u]_{\mathbb{S}_{0}^{n}}\big|^{2}\;d\nu_{n}=\int_{\mathbb{S}_{0}^{n}}|u_{n}-u_{0}|^{2}\;d\nu_{n}=\sum_{k=1}^{n}\int_{\mathbb{S}_{0}^{k}}|u_{k}-u_{k-1}|^{2}\;d\nu_{k}\label{eq:-15}
\end{equation}

Now the Poincaré inequality (\ref{eq:-35}) on $\mathbb{S}_{0}$ gives
\[
\begin{aligned}\int_{\mathbb{S}_{0}}|u_{k}(x_{1},\dots,x_{k})-u_{k-1} & (x_{1},\dots,x_{k-1})|^{2}\;\nu(dx_{k})\le C_{\ast}\,\int_{\mathbb{S}_{0}}|\nabla_{k}\,u_{k}(x_{1},\dots,x_{k})|^{2}\;\mu(dx_{k})\\
 & =C_{\ast}\,\int_{\mathbb{S}_{0}^{n-k+1}}|\nabla_{k}\,u(x_{k},\widehat{x_{k}})|^{2}\;(\mu\times\nu_{n-k+1})(dx_{k},\dots,dx_{n}),\;\;k=1,\dots,n.
\end{aligned}
\]
Therefore,
\[
\int_{\mathbb{S}_{0}^{k}}|u_{k}-u_{k-1}|^{2}\;d\nu_{k}\le C_{\ast}\,\int_{\mathbb{S}_{0}^{n}}|\nabla_{k}\,u(x_{k},\widehat{x_{k}})|^{2}\;(\mu\times\nu_{n-1})(dx_{k},d\widehat{x_{k}}),\;k=1,\dots,n.
\]
This, together with (\ref{eq:-15}), completes the proof.
\end{proof}
We can now prove the first inequality for Sobolev functions, which
is the key technical ingredient for the derivation of Sobolev inequalities
on product Sierpinski spaces.
\begin{lem}
\label{lem:-1}Let $u\in C(\mathbb{S}_{0}^{n})\cap W^{1,r}(\mathbb{S}_{0}^{n}),\;n\ge1$.
If $r>1+(n-1)\delta_{s}$ and $r\ge2$, then
\begin{equation}
\mathop{\mathrm{osc}}_{\mathbb{S}_{0}^{n}}(u)\le C_{n}\,\llbracket u\rrbracket_{W^{1,r}(\mathbb{S}_{0}^{n})}.\label{eq:-1}
\end{equation}
\end{lem}
\begin{rem}
The assumption $r\ge2$ is only to guarantee that the gradient $\nabla u$
has a proper definition. This is also the reason for having $r\ge2$
in most of the results in this paper.
\end{rem}
\begin{proof}
For any $\omega=(\omega_{1},\dots,\omega_{n})\in\mathrm{W}_{\ast}^{n}$
with $\omega_{i}=\omega_{i1}\omega_{i2}\dots\,\in\mathrm{W}_{\ast},\;1\le i\le n$,
by the Poincaré inequality (\ref{eq:}) and (\ref{eq:-4}),
\[
\begin{aligned}\int_{\mathbb{S}_{0}^{n}}\Big|u & \circ\mathbf{F}_{[\omega]_{m}}-[u\circ\mathbf{F}_{[\omega]_{m}}]_{\mathbb{S}_{0}^{n}}\Big|^{2}\,d\nu_{n}\\
 & \le C_{\ast}\sum_{i=1}^{n}\int_{\mathbb{S}_{0}^{n}}\big|\nabla_{i}\big(u\circ\mathbf{F}_{[\omega]_{m}}\big)(x_{i},\widehat{x_{i}})\big|^{2}\,(\mu\times\nu_{n-1})(dx_{i},d\widehat{x_{i}})\\
 & =C_{\ast}\Big(\frac{3}{5}\Big)^{m}\sum_{i=1}^{n}\int_{\mathbb{S}_{0}^{n}}\big|\nabla_{i}\,u\circ\mathbf{F}_{[\omega]_{m}}(x_{i},\widehat{x_{i}})\big|^{2}\,[(\mu\circ\mathbf{F}_{[\omega_{i}]_{m}})\times\nu_{n-1}](dx_{i},d\widehat{x_{i}})\\
 & =C_{\ast}\Big(\frac{3^{n}}{5}\Big)^{m}\sum_{i=1}^{n}\int_{\mathbf{F}_{[\omega]_{m}}(\mathbb{S}_{0}^{n})}\big|\nabla_{i}\,u(x_{i},\widehat{x_{i}})\big|^{2}\,(\mu\times\nu_{n-1})(dx_{i},d\widehat{x_{i}})\\
 & \le C_{\ast}\Big(\frac{3^{n}}{5}\Big)^{m}\Big[\sum_{i=1}^{n}\mu(\mathbf{F}_{[\omega_{i}]_{m}}(\mathbb{S}_{0}))\nu_{n-1}(\mathbf{F}_{[\widehat{\omega_{i}}]_{m}}(\mathbb{S}_{0}^{n-1}))\Big]^{1-2/r}\llbracket u\rrbracket_{W^{1,r}(\mathbb{S}_{0}^{n})}^{2}.
\end{aligned}
\]
Since $\mu(\mathbf{F}_{[\omega_{i}]_{m}}(\mathbb{S}_{0}))\le(3/5)^{m}$,
we obtain that
\[
\begin{aligned}\frac{1}{\nu_{n}(\mathbf{F}_{[\omega]_{m}}(\mathbb{S}_{0}^{n}))} & \int_{\mathbf{F}_{[\omega]_{m}}\big(\mathbb{S}_{0}^{n}\big)}\Big|u-[u]_{\mathbf{F}_{[\omega]_{m}}\big(\mathbb{S}_{0}^{n}\big)}\Big|\,d\nu_{n}\\
 & \le\Big[\int_{\mathbb{S}_{0}^{n}}\Big|u\circ\mathbf{F}_{[\omega]_{m}}-[u\circ\mathbf{F}_{[\omega]_{m}}]_{\mathbb{S}_{0}^{n}}\Big|^{2}\,d\nu_{n}\Big]^{1/2}\le C_{n}\,3^{-m[1/\delta_{s}-(n-1+1/\delta_{s})/r]}\;\llbracket u\rrbracket_{W^{1,r}(\mathbb{S}_{0}^{n})}.
\end{aligned}
\]

For any $x\in\mathbb{S}_{0}^{n}$, choose $\omega\in\mathrm{W}_{\ast}^{n}$
such that $\mathbf{F}_{[\omega]_{m}}(\mathbb{S}_{0}^{n})\to x$ as
$m\to\infty$. Let $[u]_{m}=[u]_{\mathbf{F}_{[\omega]_{m}}(\mathbb{S}_{0}^{n})}$
for $m\in\mathbb{N}$. Then, by the above inequality, we have
\[
\begin{aligned}\big|[u]_{m}-[u]_{m+1}\big| & \le\frac{1}{\nu_{n}(\mathbf{F}_{[\omega]_{m+1}}(\mathbb{S}_{0}^{n}))}\int_{\mathbf{F}_{[\omega]_{m+1}}\big(\mathbb{S}_{0}^{n}\big)}\Big|u-[u]_{m}\Big|\,d\nu_{n}\\
 & \le\frac{3^{n}}{\nu_{n}(\mathbf{F}_{[\omega]_{m}}(\mathbb{S}_{0}^{n}))}\int_{\mathbf{F}_{[\omega]_{m}}\big(\mathbb{S}_{0}^{n}\big)}\Big|u-[u]_{m}\Big|\,d\nu_{n}\le C_{n}\,3^{-m[1/(r^{\prime}\delta_{s})-(n-1)/r]}\,\llbracket u\rrbracket_{W^{1,r}(\mathbb{S}_{0}^{n})}.
\end{aligned}
\]
Since $1/(r^{\prime}\delta_{s})-(n-1)/r>0$, we may conclude that
\[
\big|[u]_{k}-[u]_{0}\big|\le\sum_{m=0}^{k-1}\big|[u]_{m}-[u]_{m+1}\big|\le C_{n}\,\sum_{m=0}^{\infty}3^{-m[1/(r^{\prime}\delta_{s})-(n-1)/r]}\,\llbracket u\rrbracket_{W^{1,r}(\mathbb{S}_{0}^{n})}=C_{n}\,\llbracket u\rrbracket_{W^{1,r}(\mathbb{S}_{0}^{n})}.
\]
Using the Lebesgue differentiation theorem and setting $k\to\infty$
gives
\[
\Big|u(x)-\int_{\mathbb{S}_{0}^{n}}u\;d\nu_{n}\Big|\le C_{n}\,\llbracket u\rrbracket_{W^{1,r}(\mathbb{S}_{0}^{n})},
\]
which implies (\ref{eq:-1}).
\end{proof}
\begin{rem}
When $n=1,\,p=2$, the inequality (\ref{eq:-1}) reduces to the following
well-known inequality on $\mathbb{S}_{0}$ (see, for example, \citep{Ki89})
\[
\mathop{\mathrm{osc}}_{\mathbb{S}_{0}}(u)\le C\,\mathcal{E}_{0}(u,u)^{1/2},\quad u\in\mathcal{F}(\mathbb{S}_{0}).
\]
\end{rem}
\begin{prop}
\label{prop:}Suppose that $u\in W^{1,r}(\mathbb{S}_{0}^{n}),\;r>1+(n-1)\delta_{s}$.
Then $u\in C(\mathbb{S}_{0}^{n})$ and
\begin{equation}
\mathop{\mathrm{osc}}_{\mathbb{S}_{0}^{n}}(u\circ\mathbf{F}_{[\omega]_{m}})\le C_{n}\,3^{-m\alpha_{r}}\,\llbracket u\rrbracket_{W^{1,r}(\mathbf{F}_{[\omega]_{m}}(\mathbb{S}_{0}^{n}))},\label{eq:-5}
\end{equation}
where 
\begin{equation}
\alpha_{r}=1/(r^{\prime}\delta_{s})-(n-1)/r.\label{eq:-7}
\end{equation}
\end{prop}
\begin{proof}
For any $\omega\in\mathrm{W}_{\ast}^{n}$ and any $m\in\mathbb{N}$,
by (\ref{eq:-1}),
\[
\mathop{\mathrm{osc}}_{\mathbb{S}_{0}^{n}}\big(u\circ\mathbf{F}_{[\omega]_{m}}\big)\le C_{n}\,\llbracket u\circ\mathbf{F}_{[\omega]_{m}}\rrbracket_{W^{1,r}(\mathbb{S}_{0}^{n})}.
\]
By (\ref{eq:-4}), we deduce that
\[
\begin{aligned}\llbracket u\circ\mathbf{F}_{[\omega]_{m}}\rrbracket_{W^{1,r}(\mathbb{S}_{0}^{n})}^{r} & =\sum_{i=1}^{n}\int_{\mathbb{S}_{0}^{n}}\big|\nabla_{i}\big(u\circ\mathbf{F}_{[\omega]_{m}}\big)(x_{i},\widehat{x_{i}})\big|^{r}\,(\mu\times\nu_{n-1})(dx_{i},d\widehat{x_{i}})\\
 & \le3^{-m(r-1)/\delta_{s}}\sum_{i=1}^{n}\int_{\mathbb{S}_{0}^{n}}\big|\nabla_{i}u\circ\mathbf{F}_{[\omega]_{m}}(x_{i},\widehat{x_{i}})\big|^{r}\,[(\mu\circ\mathbf{F}_{[\omega_{i}]_{m}})\times\nu_{n-1}](dx_{i},d\widehat{x_{i}})\\
 & \le3^{-m[(r-1)/\delta_{s}-n+1]}\sum_{i=1}^{n}\int_{\mathbf{F}_{[\omega]_{m}}(\mathbb{S}_{0}^{n})}\big|\nabla_{i}u(x_{i},\widehat{x_{i}})\big|^{r}\,(\mu\times\nu_{n-1})(dx_{i},d\widehat{x_{i}})\\
\vphantom{\int_{\mathbf{F}_{[\omega]_{m}}(\mathbb{S}_{0}^{n})}} & =3^{-m[(r-1)/\delta_{s}-n+1]}\llbracket u\rrbracket_{W^{1,r}(\mathbf{F}_{[\omega]_{m}}(\mathbb{S}_{0}^{n}))}^{r},
\end{aligned}
\]
which completes the proof.
\end{proof}
To proceed further, we shall need the result below on the Kusuoka
measure.
\begin{lem}
\label{lem:-2}There exists an $\omega\in\mathrm{W}_{\ast}$ such
that
\begin{equation}
\lim_{k\to\infty}\frac{\mu\circ\mathbf{F}_{1}^{m}\big(\mathbf{F}_{[\omega]_{k}}(\mathbb{S}_{0})\big)}{\mu\big(\mathbf{F}_{[\omega]_{k}}(\mathbb{S}_{0})\big)}=\inf_{\omega\in\mathrm{W}_{\ast},\;k\in\mathbb{N}_{+}}\frac{\mu\circ\mathbf{F}_{1}^{m}\big(\mathbf{F}_{[\omega]_{k}}(\mathbb{S}_{0})\big)}{\mu\big(\mathbf{F}_{[\omega]_{k}}(\mathbb{S}_{0})\big)}=\Big(\frac{1}{15}\Big)^{m}.\label{eq:-6}
\end{equation}
\end{lem}
\begin{proof}
By definition,
\[
\frac{\mu\circ\mathbf{F}_{1}^{m}\big(\mathbf{F}_{[\omega]_{k}}(\mathbb{S}_{0})\big)}{\mu\big(\mathbf{F}_{[\omega]_{k}}(\mathbb{S}_{0})\big)}=\frac{\mathrm{trace}\big((\mathbf{Y}_{1}^{m})^{\mathrm{t}}\,\mathbf{Y}_{[\omega]_{k}}^{\mathrm{t}}\mathbf{Y}_{[\omega]_{k}}\mathbf{Y}_{1}^{m}\big)}{\mathrm{trace}\big(\mathbf{Y}_{[\omega]_{k}}^{\mathrm{t}}\mathbf{Y}_{[\omega]_{k}}\big)},\;\;\omega\in\mathrm{W}_{\ast}.
\]
Since $\mathrm{spectrum}(\mathbf{Y}_{1})=\big\{0,\,\frac{1}{5},\,\frac{3}{5}\big\}$,
we deduce that
\[
\mathrm{trace}\big((\mathbf{Y}_{1}^{m})^{\mathrm{t}}\,\mathbf{Y}_{[\omega]_{k}}^{\mathrm{t}}\mathbf{Y}_{[\omega]_{k}}\mathbf{Y}_{1}^{m}\big)\ge5^{-2m}\;\mathrm{trace}\big(\mathbf{Y}_{[\omega]_{k}}^{\mathrm{t}}\mathbf{Y}_{[\omega]_{k}}\big),
\]
which implies that
\[
\inf_{\omega\in\mathrm{W}_{\ast},\;k\in\mathbb{N}_{+}}\frac{\mu\circ\mathbf{F}_{1}^{m}\big(\mathbf{F}_{[\omega]_{k}}(\mathbb{S}_{0})\big)}{\mu\big(\mathbf{F}_{[\omega]_{k}}(\mathbb{S}_{0})\big)}\ge\Big(\frac{1}{15}\Big)^{m}.
\]

To show the reverse inequality, let

\[
\mathbf{Q}=\left[\begin{array}{ccc}
\vspace{1.5mm}\frac{2}{\sqrt{6}} & 0 & \frac{1}{\sqrt{3}}\\
\vspace{1.5mm}-\frac{1}{\sqrt{6}} & \frac{1}{\sqrt{2}} & \frac{1}{\sqrt{3}}\\
\vspace{1.5mm}-\frac{1}{\sqrt{6}} & -\frac{1}{\sqrt{2}} & \frac{1}{\sqrt{3}}
\end{array}\right].
\]
Then
\[
\mathbf{Q}^{\mathrm{t}}\mathbf{Y}_{i}\mathbf{Q}=\left[\begin{array}{cc}
\mathbf{M}_{i} & 0\\
0 & 0
\end{array}\right],\quad i=1,2,3,
\]
where
\begin{equation}
\mathbf{M}_{1}=\left[\begin{array}{cc}
\vspace{1mm}\frac{3}{5} & 0\\
\vspace{1mm}0 & \frac{1}{5}
\end{array}\right],\;\mathbf{M}_{2}=\left[\begin{array}{cc}
\vspace{1mm}\frac{3}{10} & -\frac{\sqrt{3}}{10}\\
\vspace{1mm}-\frac{\sqrt{3}}{10} & \frac{1}{2}
\end{array}\right],\;\mathbf{M}_{3}=\left[\begin{array}{cc}
\vspace{1mm}\frac{3}{10} & \frac{\sqrt{3}}{10}\\
\vspace{1mm}\frac{\sqrt{3}}{10} & \frac{1}{2}
\end{array}\right].\label{eq:-16}
\end{equation}
Denote
\[
\mathbf{M}_{[\omega]_{k}}=\left[\begin{array}{cc}
M(k)_{11} & M(k)_{12}\\
M(k)_{21} & M(k)_{22}
\end{array}\right],\quad k\in\mathbb{N}_{+}.
\]
Then
\[
\begin{aligned}\frac{\mathrm{trace}\big((\mathbf{Y}_{1}^{m})^{\mathrm{t}}\,\mathbf{Y}_{[\omega]_{k}}^{\mathrm{t}}\mathbf{Y}_{[\omega]_{k}}\mathbf{Y}_{1}^{m}\big)}{\mathrm{trace}\big(\mathbf{Y}_{[\omega]_{k}}^{\mathrm{t}}\mathbf{Y}_{[\omega]_{k}}\big)} & =\frac{\mathrm{trace}\big((\mathbf{M}_{1}^{m})^{\mathrm{t}}\,\mathbf{M}_{[\omega]_{k}}^{\mathrm{t}}\mathbf{M}_{[\omega]_{k}}\mathbf{M}_{1}^{m}\big)}{\mathrm{trace}\big(\mathbf{M}_{[\omega]_{k}}^{\mathrm{t}}\mathbf{M}_{[\omega]_{k}}\big)}\\
 & =\Big(\frac{3}{5}\Big)^{2m}\;\frac{\sum_{i}M(k)_{i1}^{2}}{\sum_{i,j}M(k)_{ij}^{2}}+\Big(\frac{1}{5}\Big)^{2m}\;\frac{\sum_{i}M(k)_{i2}^{2}}{\sum_{i,j}M(k)_{ij}^{2}}.
\end{aligned}
\]

Let $\omega=2333\dots\,\in\mathrm{W}_{\ast}$. By simple computation,
\[
\mathbf{M}_{[\omega]_{k}}=\mathbf{M}_{3}^{k-1}\,\mathbf{M}_{2}=\left[\begin{array}{cc}
\vspace{2mm}\frac{3}{10}\cdot\big(\frac{1}{5}\big)^{k-1} & \frac{\sqrt{3}}{10}\cdot\big(\frac{3}{5}\big)^{k-1}-\frac{2\sqrt{3}}{10}\cdot\big(\frac{1}{5}\big)^{k-1}\\
\vspace{2mm}-\frac{\sqrt{3}}{10}\cdot\big(\frac{1}{5}\big)^{k-1} & \frac{3}{10}\cdot\big(\frac{3}{5}\big)^{k-1}+\frac{2}{10}\cdot\big(\frac{1}{5}\big)^{k-1}
\end{array}\right].
\]
Clearly,
\[
\lim_{k\to\infty}\frac{\sum_{i}M(k)_{i1}^{2}}{\sum_{i,j}M(k)_{ij}^{2}}=0,\;\lim_{k\to\infty}\frac{\sum_{i}M(k)_{i2}^{2}}{\sum_{i,j}M(k)_{ij}^{2}}=1,
\]
and therefore
\[
\lim_{k\to\infty}\frac{\mathrm{trace}\big((\mathbf{M}_{1}^{m})^{\mathrm{t}}\,\mathbf{M}_{[\omega]_{k}}^{\mathrm{t}}\mathbf{M}_{[\omega]_{k}}\mathbf{M}_{1}^{m}\big)}{\mathrm{trace}\big(\mathbf{M}_{[\omega]_{k}}^{\mathrm{t}}\mathbf{M}_{[\omega]_{k}}\big)}=\Big(\frac{1}{5}\Big)^{2m},
\]
This implies
\[
\lim_{k\to\infty}\frac{\mu\circ\mathbf{F}_{1}^{m}\big(\mathbf{F}_{[\omega]_{k}}(\mathbb{S}_{0})\big)}{\mu\big(\mathbf{F}_{[\omega]_{k}}(\mathbb{S}_{0})\big)}=\Big(\frac{1}{15}\Big)^{m}.
\]

The proof is completed.
\end{proof}
As an immediate corollary of Lemma \ref{lem:-2}, we obtain the following.
\begin{cor}
\label{cor:}Let $m\in\mathbb{N}$. Then
\begin{equation}
\Big{\Vert}\frac{d(\mu\circ\mathbf{F}_{1}^{-m})}{d\mu}\Big{\Vert}_{L^{\infty}(\mathbb{S};\mu)}=\Big{\Vert}\frac{d\mu}{d(\mu\circ\mathbf{F}_{1}^{m})}\Big{\Vert}_{L^{\infty}(\mathbb{S};\mu)}=15^{m}.\label{eq:-37}
\end{equation}
\end{cor}
The following proposition shows that a function $u\in W^{1,r}(\mathbb{S}^{n})$
with $r>1+(n-1)\delta_{s},\,r\ge2$ has at most polynomial growth.
\begin{prop}
\label{prop:-1}Suppose $u\in W^{1,r}(\mathbb{S}^{n}),\;r>1+(n-1)\delta_{s},\,r\ge2$.
Let $S_{0,m}=\mathbf{F}_{(1,\dots,1)}^{-m}(\mathbb{S}_{0}^{n}),\,m\in\mathbb{N}$.
Then
\begin{equation}
\mathop{\mathrm{osc}}_{S_{0,m}}\big(u\big)\le3^{m\beta_{r}}\;\llbracket u\rrbracket_{W^{1,r}(S_{0,m})},\label{eq:-8}
\end{equation}
where
\begin{equation}
\beta_{r}=(1/\delta_{s}+1)/r^{\prime}-n/r.\label{eq:-12}
\end{equation}
\end{prop}
\begin{proof}
By (\ref{eq:-2}) ,
\[
\begin{aligned}\nabla_{i}\,u & =\nabla_{i}\big(u\circ\mathbf{F}_{(1,\dots,1)}^{-m}\circ\mathbf{F}_{(1,\dots,1)}^{m}\big)\\
 & =\Big(\frac{3}{5}\Big)^{m/2}\;\nabla_{i}\big(u\circ\mathbf{F}_{(1,\dots,1)}^{-m}\big)\circ\mathbf{F}_{(1,\dots,1)}^{m}\cdot\bigg[\frac{d(\mu\circ\mathbf{F}_{1}^{m})}{d\mu}\bigg]^{1/2}.
\end{aligned}
\]
Therefore,
\[
\nabla_{i}\big(u\circ\mathbf{F}_{(1,\dots,1)}^{-m}\big)=3^{m/(2\delta_{s})}\;\nabla_{i}\,u\circ\mathbf{F}_{(1,\dots,1)}^{-m}\cdot\bigg[\frac{d(\mu\circ\mathbf{F}_{1}^{-m})}{d\mu}\bigg]^{1/2}.
\]
It follows from the above and Corollary \ref{cor:} that
\[
\big|\nabla_{i}\big(u\circ\mathbf{F}_{(1,\dots,1)}^{-m}\big)\big|\le3^{m[(1/\delta_{s}+2)/r^{\prime}-1]}\;\big|\nabla_{i}\,u\circ\mathbf{F}_{(1,\dots,1)}^{-m}\big|\cdot\bigg[\frac{d(\mu\circ\mathbf{F}_{1}^{-m})}{d\mu}\bigg]^{1/r}.
\]
Therefore,
\[
\begin{aligned}\int_{\mathbb{S}_{0}^{n}}\big|\nabla_{i}\big(u & \circ\mathbf{F}_{(1,\dots,1)}^{-m}\big)(x_{i},\widehat{x_{i}})\big|^{r}\,(\mu\times\nu_{n-1})(dx_{i},d\widehat{x_{i}})\\
 & \le3^{m[(1/\delta_{s}+2)(r-1)-r]}\int_{\mathbb{S}_{0}^{n}}\big|\nabla_{i}\,u\circ\mathbf{F}_{(1,\dots,1)}^{-m}(x_{i},\widehat{x_{i}})\big|^{r}\,[(\mu\circ\mathbf{F}_{1}^{-m})\times\nu_{n-1}](dx_{i},d\widehat{x_{i}})\\
 & =3^{m[(1/\delta_{s}+1)(r-1)-n]}\int_{S_{0,m}}\big|\nabla_{i}\,u(x_{i},\widehat{x_{i}})\big|^{r}\,(\mu\times\nu_{n-1})(dx_{i},d\widehat{x_{i}}),
\end{aligned}
\]
which yields that
\[
\big{\llbracket}u\circ\mathbf{F}_{(1,\dots,1)}^{-m}\big{\rrbracket}_{W^{1,r}(\mathbb{S}_{0}^{n})}\le3^{m\beta_{r}}\;\llbracket u\rrbracket_{W^{1,r}(S_{0,m})}.
\]
Now (\ref{eq:-8}) follows immediately from (\ref{eq:-1}) and the
above inequality.
\end{proof}
To formulate our first main result, let us first introduce the setting
that we shall work on. Let $\sigma$ be a positive Radon measure on
$\mathbb{S}^{n}$ satisfying the following condition: there exist
constants $0<\underbar{\ensuremath{\delta}}\le\overline{\delta}\le\infty$
with $\overline{\delta}\ge1$ and $C_{\sigma}>0$ such that
\begin{equation}
\left\{ \begin{aligned}\sigma(S)\le C_{\sigma}\nu_{n}(S)^{1/\overline{\delta}}, & \quad\text{if}\ 0<\mathrm{diam}(S)\le1,\\
\sigma(S)\le C_{\sigma}\nu_{n}(S)^{1/\underbar{\ensuremath{{\scriptstyle \delta}}}}, & \quad\quad\text{if}\ \mathrm{diam}(S)>1
\end{aligned}
\right.\tag{{M}}\label{eq:-9}
\end{equation}
for all dyadic simplexes $S\subseteq\mathbb{S}^{n}$.
\begin{rem}
Note that the restriction $\overline{\delta}\ge1$ in (\ref{eq:-9})
is necessary in view of the countable additivity of $\sigma$ and
$\nu_{n}$ and that $\sigma$ is finite on compact subsets.
\end{rem}
We list some examples of the Radon measure $\sigma$.
\begin{example}
(i) The Hausdorff measure $\nu_{n}$, for which $\underbar{\ensuremath{\delta}}=\overline{\delta}=1$.

(ii) The product Kusuoka measure $\mu_{n}=\mu\times\cdots\times\mu$,
for which the sharp constants $\underbar{\ensuremath{\delta}}$ and
$\overline{\delta}$ will be given later. (See Corollary \ref{cor:-1}-(a).)

(iii) Dirac measures, for which $\underbar{\ensuremath{\delta}}=\overline{\delta}=\infty$.

(iv) Examples (ii) and (iii) can be generalized to linear combinations
of measures of the form $\sigma=\sigma_{1}\times\sigma_{2}$, where
$\sigma_{1},\,\sigma_{2}$ are Radon measures on $\mathbb{S}^{k}$
and $\mathbb{S}^{n-k}$ satisfying conditions (\ref{eq:-9}) on the
corresponding spaces. Another particular case of such measures is
\[
\sigma=\sum_{i=1}^{n}\sum_{j=1,2,3}\nu_{i-1}\times\delta_{p_{j}}\times\nu_{n-i},
\]
where $\delta_{p_{j}}$ is the Dirac measure concentrated at $p_{j}\in\mathrm{V}_{0,0}$.
Applying Theorem \ref{thm:} below to the above measure $\sigma$
gives the trace theorem for functions in $W^{1,r}(\mathbb{S}_{0}^{n})$.
\end{example}
We are now in a position to formulate the Sobolev inequalities on
the infinite product space $\mathbb{S}^{n}$.
\begin{thm}
\label{thm:}Suppose $\sigma$ is a positive Radon measure on $\mathbb{S}^{n}$
satisfying the condition (\ref{eq:-9}). Let $r>1+(n-1)\delta_{s},\,r\ge2,\;p\ge1,\;\min\{p,r\}\le q\le\infty$.
Then
\begin{equation}
\Vert u\Vert_{L^{q}(\sigma)}\le C\,\sum_{i=1,2}\llbracket u\rrbracket_{W^{1,r}(\mathbb{S}^{n})}^{a_{i}}\Vert u\Vert_{L^{p}(\nu_{n})}^{1-a_{i}},\label{eq:-14}
\end{equation}
where
\begin{equation}
a_{1}=\Big[\frac{1/p-1/(q\underbar{\ensuremath{\delta}})}{1/p-1/r+[1/(r^{\prime}\delta_{s})+1/r^{\prime}]/n}\Big]^{+},\;a_{2}=\Big[\frac{1/p-1/(q\overline{\delta})}{1/p-1/r+[1/(r^{\prime}\delta_{s})+1/r]/n}\Big]^{+},\label{eq:-13}
\end{equation}
and $C>0$ is a constant depending only on the constant $C_{\sigma}$
in (\ref{eq:-9}).
\end{thm}
\begin{proof}
Let $\mathbb{S}_{i}^{n}=\mathbb{S}_{i_{1}}\times\cdots\times\mathbb{S}_{i_{n}},\,i=(i_{1},\dots,i_{n})\in\mathbb{Z}^{n}$
be non-overlapping translations of $\mathbb{S}_{0}^{n}$ such that
$\mathbb{S}^{n}=\bigcup_{i}\mathbb{S}_{i}^{n}$ (the specific order
of $\mathbb{S}_{i}^{n}$ does not matter). For any $m\in\mathbb{Z}$,
let 
\[
S_{i,m}=\mathbf{F}_{(1,\dots,1)}^{-m}(\mathbb{S}_{i}^{n})=(\mathbf{F}_{1}^{\otimes n})^{-m}(\mathbb{S}_{i}^{n}).
\]
Then $S_{i,m}$ are dyadic simplexes with $\mathrm{diam}(S_{i,m})=2^{m}$,
and $\mathbb{S}^{n}=\bigcup_{i}S_{i,m}$. Denote
\[
\nu_{n}(m)=\nu_{n}(S_{0,m})=3^{mn},\;[u]_{S_{i,m}}=\frac{1}{\nu_{n}(m)}\int_{S_{i,m}}u\;d\nu_{n},\;\;i,\,m\in\mathbb{Z}.
\]

\medskip{}

For $m\ge0$, by Proposition \ref{prop:-1} and (\ref{eq:-9}),
\[
\begin{aligned}\Vert u\Vert_{L^{q}(\sigma)}^{q} & \le2^{q-1}\,\sum_{i}\Big[\int_{S_{i,m}}\big|u-[u]_{S_{i,m}}\big|^{q}\;d\sigma+\sigma(S_{i,m})\,\big|[u]_{S_{i,m}}\big|^{q}\Big]\\
 & \le2^{q-1}\,\sum_{i}\Big[\sigma(S_{i,m})\nu_{n}(m)^{\beta_{r}q/n}\llbracket u\rrbracket_{W^{1,r}(S_{i,m})}^{q}+\sigma(S_{i,m})\,\nu_{n}(m)^{-q/p}\Big(\int_{S_{i,m}}|u|^{p}\;d\nu_{n}\Big)^{q/p}\Big]\\
 & \le C_{\sigma}\,2^{q-1}\,\sum_{i}\Big[\nu_{n}(m)^{1/\underbar{\ensuremath{{\scriptstyle \delta}}}+\beta_{r}q/n}\llbracket u\rrbracket_{W^{1,r}(S_{i,m})}^{q}+\nu_{n}(m)^{1/\underbar{\ensuremath{{\scriptstyle \delta}}}-q/p}\,\Big(\int_{S_{i,m}}|u|^{p}\;d\nu_{n}\Big)^{q/p}\Big]\\
 & \le C_{\sigma}\,2^{q-1}\,\Big[\nu_{n}(m)^{1/\underbar{\ensuremath{{\scriptstyle \delta}}}+\beta_{r}q/n}\Big(\sum_{i}\llbracket u\rrbracket_{W^{1,r}(S_{i,m})}^{r}\Big)^{q/r}+\nu_{n}(m)^{1/\underbar{\ensuremath{{\scriptstyle \delta}}}-q/p}\,\Big(\sum_{i}\int_{S_{i,m}}|u|^{p}\;d\nu_{n}\Big)^{q/p}\Big]\\
 & \le C_{\sigma}\,2^{q-1}\,\Big[\nu_{n}(m)^{1/\underbar{\ensuremath{{\scriptstyle \delta}}}+\beta_{r}q/n}\Big(\sum_{i}\llbracket u\rrbracket_{W^{1,r}(S_{i,m})}^{r}\Big)^{q/r}+\nu_{n}(m)^{1/\underbar{\ensuremath{{\scriptstyle \delta}}}-q/p}\,\Big(\sum_{i}\int_{S_{i,m}}|u|^{p}\;d\nu_{n}\Big)^{q/p}\Big]\\
\vphantom{\int_{S_{i,m}}} & \le C_{\sigma}\,2^{q-1}\,\Big[\nu_{n}(m)^{1/\underbar{\ensuremath{{\scriptstyle \delta}}}+\beta_{r}q/n}\llbracket u\rrbracket_{W^{1,r}(\mathbb{S}^{n})}^{q}+\nu_{n}(m)^{1/\underbar{\ensuremath{{\scriptstyle \delta}}}-q/p}\,\Vert u\Vert_{L^{p}(\nu_{n})}^{q}\Big],
\end{aligned}
\]
where $\alpha_{r},\,\beta_{r}$ be the exponents given by (\ref{eq:-7})
and (\ref{eq:-12}) respectively.

Therefore,
\begin{equation}
\Vert u\Vert_{L^{q}(\sigma)}\le C\,\Big[\nu_{n}(m)^{1/(q\underbar{\ensuremath{{\scriptstyle \delta}}})+\beta_{r}/n}\,\llbracket u\rrbracket_{W^{1,r}(\mathbb{S}^{n})}+\nu_{n}(m)^{1/(q\underbar{\ensuremath{{\scriptstyle \delta}}})-1/p}\,\Vert u\Vert_{L^{p}(\nu_{n})}\Big],\label{eq:-10}
\end{equation}
where $C>0$ is a constant depending only on $C_{\sigma}$ (the constant
$C$ can chosen to be independent of $q$ as $q>1$).

\medskip{}

Similarly, for $m\le0$, we have
\begin{equation}
\Vert u\Vert_{L^{q}(\sigma)}\le C\,\Big[\nu_{n}(m)^{1/(q\overline{\delta})+\alpha_{r}/n}\,\llbracket u\rrbracket_{W^{1,r}(\mathbb{S}^{n})}+\nu_{n}(m)^{1/(q\overline{\delta})-1/p}\,\Vert u\Vert_{L^{p}(\nu_{n})}\Big].\label{eq:-11}
\end{equation}

\medskip{}

Without loss of generality, we may assume that $\llbracket u\rrbracket_{W^{1,r}(\mathbb{S}^{n})}>0$.

\smallskip{}

For the case $\llbracket u\rrbracket_{W^{1,r}(\mathbb{S}^{n})}\le\Vert u\Vert_{L^{p}(\nu_{n})}$,
setting
\[
m=\inf\Big\{ m\ge0:\nu_{n}(m)^{\beta_{r}/n+1/p}\ge\Vert u\Vert_{L^{p}(\nu_{n})}/\llbracket u\rrbracket_{W^{1,r}(\mathbb{S}^{n})}\Big\}
\]
in (\ref{eq:-10}) gives
\[
\Vert u\Vert_{L^{q}(\sigma)}\le C\,\llbracket u\rrbracket_{W^{1,r}(\mathbb{S}^{n})}^{a_{1}}\,\Vert u\Vert_{L^{p}(\nu_{n})}^{1-a_{1}}.
\]

For the case $\llbracket u\rrbracket_{W^{1,r}(\mathbb{S}^{n})}>\Vert u\Vert_{L^{p}(\nu_{n})}$,
setting
\[
m=\sup\Big\{ m\le0:\nu_{n}(m)^{\beta_{r}/n+1/p}\le\Vert u\Vert_{L^{p}(\nu_{n})}/\llbracket u\rrbracket_{W^{1,r}(\mathbb{S}^{n})}\Big\}
\]
in (\ref{eq:-11}) gives
\[
\Vert u\Vert_{L^{q}(\sigma)}\le C\,\llbracket u\rrbracket_{W^{1,r}(\mathbb{S}^{n})}^{a_{2}}\,\Vert u\Vert_{L^{p}(\nu_{n})}^{1-a_{2}}.
\]

This completes the proof.
\end{proof}
\begin{rem}
(i) Recall that the Sobolev inequality on $\mathbb{R}^{n}$ takes
the form 
\begin{equation}
\Vert u\Vert_{L^{q}(\mathbb{R}^{n})}\le\Vert\nabla u\Vert_{L^{r}(\mathbb{R}^{n})}^{a}\Vert u\Vert_{L^{p}(\mathbb{R}^{n})}^{1-a},\label{eq:-42}
\end{equation}
where $a\in[0,1]$ is given by
\begin{equation}
a=\frac{1/p-1/q}{1/p-1/r+1/n}.\label{eq:-41}
\end{equation}
Let us compare (\ref{eq:-14}) and (\ref{eq:-42}). That the exponent
$q$ in (\ref{eq:-41}) is changed to the two exponents $q\underbar{\ensuremath{\delta}}$
and $q\overline{\delta}$. This is due to the different scaling rates
of the measures $\nu_{n}$ and $\sigma$, and the inhomogeneity of
the scaling rate of $\sigma$ under shrinkage and expansion. The factor
before $1/n$ in (\ref{eq:-41}) is changed from $1$ to the pair
$1/(r^{\prime}\delta_{s})+1/r^{\prime}$ and $1/(r\delta_{s})+1/r^{\prime}$.
When $r=2$, these numbers are both equal to $1/(2\delta_{s})+1/2$,
which is the natural factor for $\mathbb{S}$ as $\delta_{s}=1$ if
the spectral dimension were $d_{s}=1$. When $r>2$, these numbers
depend on the exponent $r$. This suggests that $r-2$, the excessing
part of $r$, has an distorting effect on the dimension $n$. Such
distorting effect can also be seen from (\ref{eq:-2}) and (\ref{eq:-4}).

(ii) Only the term corresponding to $a_{1}$ is needed on the right
hand side of (\ref{eq:-14}) if $\mathbb{S}^{n}$ is replaced by $\mathbb{S}_{0}^{n}$,
since only the first part of (\ref{eq:-9}) is involved. The proof
of this is similar to that of Theorem \ref{thm:} and hence omitted.
\end{rem}
We now show that the condition (\ref{eq:-9}) is also necessary for
the Sobolev inequality (\ref{eq:-14}) to hold.
\begin{thm}
\label{thm:-1}Suppose that $\sigma$ is a positive Radon measure
on $\mathbb{S}^{n}$, and there exist constants $p\ge1,\;1\le q<\infty,\;r>1+(n-1)\delta_{s},\,r\ge2,\;a_{i}\in[0,1],\,1\le i\le k$
and $C>0$ such that
\begin{equation}
\Vert u\Vert_{L^{q}(\sigma)}\le C\,\sum_{i=1}^{k}\llbracket u\rrbracket_{W^{1,r}(\mathbb{S}^{n})}^{a_{i}}\Vert u\Vert_{L^{p}(\nu_{n})}^{1-a_{i}}\;\;\text{for all}\ u\in W^{1,r}(\mathbb{S}^{n}).\label{eq:-39}
\end{equation}
Then $\sigma$ satisfies the condition (\ref{eq:-9}) for some $0<\underbar{\ensuremath{\delta}}\le\overline{\delta}\le\infty$
with $\overline{\delta}\ge1$ and $C_{\sigma}>0$.
\end{thm}
\begin{proof}
For the first part of (\ref{eq:-9}), it suffices to show that $\sup_{i\in\mathbb{Z}^{n}}\sigma(\mathbb{S}_{i}^{n})<\infty$
and take $\overline{\delta}=\infty$. (Note that $\overline{\delta}=\infty$
is the only valid value when $\sigma$ is a Dirac measure.) Let $\phi_{0}$
be the $1$-harmonic function in $\mathbb{S}_{0}$ with boundary value
$\phi_{0}|_{\mathrm{V}_{0,1}}=1_{\mathbf{F}_{2}(p_{1})}$. Clearly,
$\phi_{0}|_{\mathrm{V}_{0,0}}=0$. Therefore, setting $\phi_{0}=0$
on $\mathbb{S}\backslash\mathbb{S}_{0}$ gives a function $\phi_{0}\in\mathcal{F}(\mathbb{S})$.
It is easily seen that
\[
\frac{2}{5}\le\phi_{0}\le1\ \text{on}\ \mathbf{F}_{1}\circ\mathbf{F}_{2}(\mathbb{S}_{0}),\;\;\mathrm{supp}(\phi_{0})\subseteq\mathbb{S}_{0}.
\]
Note that $(\phi_{0}\circ\mathbf{F}_{1})|_{\mathbb{S}_{0}}$ is the
harmonic function in $\mathbb{S}_{0}$ with boundary value $(\phi_{0}\circ\mathbf{F}_{1})\big|_{\mathrm{V}_{0,0}}=1_{\{p_{2}\}}$.
By (\ref{eq:-38}), (\ref{eq:-2}) and (\ref{eq:-37}),
\[
\Vert\nabla\phi_{0}\Vert_{L^{\infty}(\mathbb{S};\mu)}=\Vert\nabla\phi_{0}\circ\mathbf{F}_{1}\Vert_{L^{\infty}(\mathbb{S};\mu)}\le\sqrt{\frac{5}{3}}\;\;\big{\Vert}\nabla(\phi_{0}\circ\mathbf{F}_{1})\big{\Vert}_{L^{\infty}(\mathbb{S}_{0};\mu)}\;\Big{\Vert}\frac{d\mu}{d(\mu\circ\mathbf{F}_{1})}\Big{\Vert}_{L^{\infty}(\mathbb{S};\mu)}^{1/2}\le5\sqrt{2}.
\]
Therefore,
\[
\llbracket\nabla\phi_{0}\rrbracket_{W^{1,r}(\mathbb{S})}\le5\sqrt{2}.
\]
For any $j\in\mathbb{Z}$, let $\phi_{j}\in\mathcal{F}(\mathbb{S})$
be the translation of $\phi_{0}$ such that $\mathrm{supp}(\phi_{j})\subseteq\mathbb{S}_{j}$.
For each $i=(i_{1},\dots,i_{n})\in\mathbb{Z}^{n}$, let
\[
u_{i}(x)=\phi_{i_{1}}(x_{1})\cdots\phi_{i_{n}}(x_{n}),\;\;x=(x_{1},\dots,x_{n})\in\mathbb{S}^{n}.
\]
Then $u_{i}\in W^{1,r}(\mathbb{S}^{n}),\;\mathrm{supp}(u_{i})\subseteq\mathbb{S}_{i}^{n}$.
Let $\tau_{i},\,i\in\mathbb{Z}^{n}$ be the translations mapping $\mathbb{S}_{0}^{n}$
onto $\mathbb{S}_{i}^{n}$, and 
\[
S_{i}=\tau_{i}(\mathbf{F}_{(1,\dots,1)}\circ\mathbf{F}_{(2,\dots,2)}\mathbb{S}_{0}^{n}).
\]
Then
\[
\Big(\frac{2}{5}\Big)^{n}\le u_{i}\le1\ \text{on}\ S_{i},\;\;\llbracket u_{i}\rrbracket_{W^{1,r}(\mathbb{S}^{n})}\le5\sqrt{2}\;n^{1/r}.
\]

Setting $u=u_{i}$ in (\ref{eq:-39}) gives $\sup_{i\in\mathbb{Z}^{n}}\sigma(S_{i})^{1/q}<\infty$.
Since $q<\infty$, it is seen that $\sup_{i\in\mathbb{Z}^{n}}\sigma(S_{i})<\infty$.
Clearly, this implies $\sup_{i\in\mathbb{Z}^{n}}\sigma(\mathbb{S}_{i}^{n})<\infty$,
as we may change the boundary value of $\phi_{0}$ to $\phi_{0}|_{\mathrm{V}_{0,1}}=1_{\mathbf{F}_{3}(p_{2})}$
and $\phi_{0}|_{\mathrm{V}_{0,1}}=1_{\mathbf{F}_{1}(p_{3})}$ and
similar results hold.

\bigskip{}

We now prove the second part of (\ref{eq:-9}). Let $u_{i}$ and $S_{i}$
be the functions and simplexes defined above. It suffices to consider
simplexes $S$ of the form $S=\mathbf{F}_{(1,\dots,1)}^{-k}(S_{i}),\,k\in\mathbb{N}_{+},\,i\in\mathbb{Z}^{n}$.
The desired conclusion follows readily by changing the boundary value
of functions. 

By (\ref{eq:-4}),
\begin{equation}
\llbracket u_{i}\circ\mathbf{F}_{(1,\dots,1)}^{k}\rrbracket_{W^{1,r}(\mathbb{S}^{n})}\le3^{k\beta_{r}}\llbracket u_{i}\rrbracket_{W^{1,r}(\mathbb{S}^{n})}\le3^{k\beta_{r}}\cdot5\sqrt{2}\;n^{1/2}\;\;\text{for all}\ k\in\mathbb{N}_{+},\label{eq:-40}
\end{equation}
where $\beta_{r}>0$ is the constant given by (\ref{eq:-12}).

Let $S\subseteq\mathbb{S}^{n}$ be a simplex of the form $S=\mathbf{F}_{(1,\dots,1)}^{-k}(S_{i}),\,k\in\mathbb{N}_{+},\,i\in\mathbb{Z}^{n}$.
By (\ref{eq:-40}), setting $u=u_{i}\circ\mathbf{F}_{(1,\dots,1)}^{k}$
in (\ref{eq:-39}) gives
\[
\sigma(S)^{1/q}\le c_{n}\,3^{k(n/p+\beta_{r})}=c_{n}\,\nu_{n}(S)^{1/p+\beta_{r}/n},
\]
where $c_{n}>0$ is a constant depending only on $n$ and the constant
$C$ in (\ref{eq:-39}). Therefore, the second part of (\ref{eq:-9})
holds with $\underbar{\ensuremath{\delta}}=q(1/p+\beta_{r}/n)>0$.

This completes the proof.
\end{proof}
We may exchange the positions of $\sigma$ and $\nu$ in (\ref{eq:-14}).
Let $\sigma$ be a Radon measure satisfying the following condition:
there exist constants $0<\underbar{\ensuremath{\delta}}\le\overline{\delta}\le\infty$
with $\underbar{\ensuremath{\delta}}\le1$ and $C_{\sigma}>0$ such
that
\begin{equation}
\left\{ \begin{aligned}C_{\sigma}^{-1}\nu_{n}(S)^{1/\underbar{\ensuremath{{\scriptstyle \delta}}}}\le\sigma(S), & \quad\text{if}\ 0<\mathrm{diam}(S)\le1,\\
C_{\sigma}^{-1}\nu_{n}(S)^{1/\overline{\delta}}\le\sigma(S), & \quad\quad\text{if}\ \mathrm{diam}(S)>1
\end{aligned}
\right.\tag{{\ensuremath{\mathrm{M'}}}}\label{eq:-36}
\end{equation}
for all dyadic simplexes $S\subseteq\mathbb{S}^{n}$. Note that, as
in (\ref{eq:-9}), the restriction $\underbar{\ensuremath{\delta}}\le1$
is necessary in view of the countable additivity of measures.
\begin{rem}
The sharp constants for (\ref{eq:-36}) when $\sigma$ is the product
Kusuoka measure will be given later. (See Corollary \ref{cor:-1}-(b).)
\end{rem}
For such Radon measures, we have the following two theorems, of which
the proofs are similar to those of Theorem \ref{thm:} and Theorem
\ref{thm:-1}, and hence will be omitted.
\begin{thm}
\label{thm:-2}Suppose $\sigma$ is a positive Radon measure on $\mathbb{S}^{n}$
satisfying the condition (\ref{eq:-36}). Let $r>1+(n-1)\delta_{s},\,r\ge2,\;p\ge1,\;\min\{p,r\}\le q\le\infty$.
Then
\[
\Vert u\Vert_{L^{q}(\nu_{n})}\le C\,\sum_{i=1,2}\llbracket u\rrbracket_{W^{1,r}(\mathbb{S}^{n})}^{a_{i}}\Vert u\Vert_{L^{p}(\sigma)}^{1-a_{i}},
\]
where
\[
a_{1}=\Big[\frac{1/(p\overline{\delta})-1/q}{1/(p\overline{\delta})-1/r+[1/(r^{\prime}\delta_{s})+1/r^{\prime}]/n}\Big]^{+},\;a_{2}=\Big[\frac{1/(p\underbar{\ensuremath{\delta}})-1/q}{1/(p\underbar{\ensuremath{\delta}})-1/r+[1/(r^{\prime}\delta_{s})+1/r]/n}\Big]^{+},
\]
and $C>0$ is a constant depending only on the constant $C_{\sigma}$
in (\ref{eq:-36}).
\end{thm}
\begin{thm}
\label{thm:-3}Suppose $\sigma$ is a positive Radon measure on $\mathbb{S}^{n}$,
and there exist constants $p\ge1,\;1\le q<\infty,\;r>1+(n-1)\delta_{s},\,r\ge2,\;a_{i}\in[0,1],\,1\le i\le k$
and $C>0$ such that
\[
\Vert u\Vert_{L^{q}(\nu_{n})}\le C\,\sum_{i=1}^{k}\llbracket u\rrbracket_{W^{1,r}(\mathbb{S}^{n})}^{a_{i}}\Vert u\Vert_{L^{p}(\sigma)}^{1-a_{i}}\;\;\text{for all}\ u\in W^{1,r}(\mathbb{S}^{n}).
\]
Then $\sigma$ satisfies the condition (\ref{eq:-36}) for some $0<\underbar{\ensuremath{\delta}}\le\overline{\delta}\le\infty$
with $\underbar{\ensuremath{\delta}}\le1$ and $C_{\sigma}>0$.
\end{thm}

\section{\label{sec:-2}Sharp exponents for product Kusuoka measures}

We now give the sharp values of the exponents $\underbar{\ensuremath{\delta}}$
and $\bar{\delta}$ in (\ref{eq:-9}) and (\ref{eq:-36}) for the
product Kusuoka measure $\mu_{n}=\mu\times\cdots\times\mu$.
\begin{prop}
\label{prop:-3}The following holds
\begin{equation}
\inf_{\omega\in\mathrm{W}_{\ast}}\liminf_{m\to\infty}\big[\mathrm{trace}\big(\mathbf{Y}_{\omega_{1}}^{\mathrm{t}}\cdots\mathbf{Y}_{\omega_{m}}^{\mathrm{t}}\mathbf{Y}_{\omega_{m}}\cdots\mathbf{Y}_{\omega_{1}}\big)\big]^{1/m}=\frac{3}{25}.\label{eq:-43}
\end{equation}
Moreover, the infimum (\ref{eq:-43}) can be achieved by some $\omega\in\mathrm{W}_{\ast}$.
\end{prop}
\begin{proof}
Let $\mathbf{M}_{i},\,i=1,2,3$ be the $2\times2$ matrices given
by (\ref{eq:-16}). As seen in the proof of Lemma \ref{lem:-2}, it
suffices to prove the lemma for $\mathbf{M}_{i},\,i=1,2,3$.

Note that
\[
\det(\mathbf{M}_{1})=\det(\mathbf{M}_{2})=\det(\mathbf{M}_{3})=\frac{3}{25}.
\]
We see that
\[
\det\big(\mathbf{M}_{\omega_{1}}^{\mathrm{t}}\cdots\mathbf{M}_{\omega_{m}}^{\mathrm{t}}\mathbf{M}_{\omega_{m}}\cdots\mathbf{M}_{\omega_{1}}\big)=\prod_{i=1}^{m}\det\big(\mathbf{M}_{\omega_{i}}\big)^{2}=\Big(\frac{3}{25}\Big)^{2m}.
\]
Since $\mathbf{M}_{\omega_{1}}^{\mathrm{t}}\cdots\mathbf{M}_{\omega_{m}}^{\mathrm{t}}\mathbf{M}_{\omega_{m}}\cdots\mathbf{M}_{\omega_{1}}$
is non-negative definite, by the arithmetic-geometric mean inequality,
\[
2^{-1}\cdot\mathrm{trace}\big(\mathbf{M}_{\omega_{1}}^{\mathrm{t}}\cdots\mathbf{M}_{\omega_{m}}^{\mathrm{t}}\mathbf{M}_{\omega_{m}}\cdots\mathbf{M}_{\omega_{1}}\big)\ge\big[\det\big(\mathbf{M}_{\omega_{1}}^{\mathrm{t}}\cdots\mathbf{M}_{\omega_{m}}^{\mathrm{t}}\mathbf{M}_{\omega_{m}}\cdots\mathbf{M}_{\omega_{1}}\big)\big]^{1/2}=\Big(\frac{3}{25}\Big)^{m}.
\]
This implies that
\[
\inf_{\omega\in\mathrm{W}_{\ast}}\liminf_{m\to\infty}\big[\mathrm{trace}\big(\mathbf{M}_{\omega_{1}}^{\mathrm{t}}\cdots\mathbf{M}_{\omega_{m}}^{\mathrm{t}}\mathbf{M}_{\omega_{m}}\cdots\mathbf{M}_{\omega_{1}}\big)\big]^{1/m}\ge\frac{3}{25}.
\]

To show the reverse, we construct a finite sequence $(\omega_{1}\dots\omega_{r})\in\{1,2,3\}^{r}$
such that $\mathbf{M}_{\omega_{r}}\cdots\mathbf{M}_{\omega_{1}}$
has no real eigenvalues. Once such a finite sequence can be found,
since $\mathbf{M}_{i},\,i=1,2,3$ are $2\times2$ matrices, the spectrum
of $\mathbf{M}_{\omega_{r}}\cdots\mathbf{M}_{\omega_{1}}$ must consist
of a pair of conjugate eigenvalues $\lambda,\,\bar{\lambda}\in\mathbb{C}$.
Therefore, $\mathbf{M}_{\omega_{r}}\cdots\mathbf{M}_{\omega_{1}}$
is similar to the diagonal matrix
\[
\left[\begin{array}{cc}
\lambda & 0\\
0 & \bar{\lambda}
\end{array}\right].
\]
Moreover, we have
\[
|\lambda|=\Big(\frac{3}{25}\Big)^{r/2}
\]
in view of
\[
\det\big(\mathbf{M}_{\omega_{r}}\cdots\mathbf{M}_{\omega_{1}}\big)=\prod_{i=1}^{r}\det\big(\mathbf{M}_{\omega_{i}}\big)=\Big(\frac{3}{25}\Big)^{r}.
\]
Now set $\omega$ to be the repetition
\[
\omega=(\omega_{1}\dots\omega_{r})(\omega_{1}\dots\omega_{r})\dots\,\in\mathrm{W}_{\ast}.
\]
Then, with $m=rk,\,k\in\mathbb{N}_{+}$,
\[
\mathrm{trace}\big(\mathbf{M}_{\omega_{1}}^{\mathrm{t}}\cdots\mathbf{M}_{\omega_{m}}^{\mathrm{t}}\mathbf{M}_{\omega_{m}}\cdots\mathbf{M}_{\omega_{1}}\big)=|\lambda|^{2k}+|\bar{\lambda}|^{2k}=2\cdot\Big(\frac{3}{25}\Big)^{rk}=2\cdot\Big(\frac{3}{25}\Big)^{m},\;k\in\mathbb{N}_{+}.
\]
This implies that
\[
\lim_{m\to\infty}\big[\mathrm{trace}\big(\mathbf{M}_{\omega_{1}}^{\mathrm{t}}\cdots\mathbf{M}_{\omega_{m}}^{\mathrm{t}}\mathbf{M}_{\omega_{m}}\cdots\mathbf{M}_{\omega_{1}}\big)\Big]^{1/m}=\frac{3}{25}.
\]

By a direct search, it can be seen that the least possible value of
$r$ is $r=3$, and accordingly, $(312)$ is a finite sequence satisfying
the desired property. In particular, the product

\[
\mathbf{M}_{2}\mathbf{M}_{1}\mathbf{M}_{3}=\left[\begin{array}{cc}
\vspace{1mm}\frac{6}{125} & \frac{\sqrt{3}}{125}\\
\vspace{1mm}-\frac{\sqrt{3}}{125} & \frac{4}{125}
\end{array}\right]
\]
is a such product having eigenvalues $\frac{1}{25}\big(1\pm\frac{\sqrt{2}}{5}\mathrm{i}\big).$

We may now take $\omega=(312)(312)\dots\,\in\mathrm{W}_{\ast}$ and
complete the proof.
\end{proof}
\begin{cor}
\label{cor:-1}(a) The measure $\sigma=\mu_{n}$ satisfies the condition
(\ref{eq:-9}) with $\underbar{\ensuremath{\delta}}=1$ and $\overline{\delta}=\delta_{s}$.
Conversely, if $\mu_{n}$ satisfies (\ref{eq:-9}) for some $\underbar{\ensuremath{\delta}}$
and $\overline{\delta}$, then $\underbar{\ensuremath{\delta}}\le1,\;\overline{\delta}\ge\delta_{s}$.

(b) The measure $\sigma=\mu_{n}$ satisfies the condition (\ref{eq:-36})
with $\underbar{\ensuremath{\delta}}=(1+1/\delta_{s})^{-1}$ and $\overline{\delta}=1$.
Conversely, if $\mu_{n}$ satisfies (\ref{eq:-36}) for some $\underbar{\ensuremath{\delta}}$
and $\overline{\delta}$, then $\underbar{\ensuremath{\delta}}\le(1+1/\delta_{s})^{-1},\;\overline{\delta}\ge1$.
\end{cor}
\begin{proof}
(a) We only need to prove the statements on $\overline{\delta}$,
as $\mu_{n}(S)=\nu_{n}(S)=3^{k}$ for dyadic simplexes $S\subseteq\mathbb{S}^{n}$
with $\mathrm{diam}(S)=2^{k},\,k\in\mathbb{N}$. Since $\mathrm{spectrum}(\mathbf{Y}_{i})=\big\{0,\,\frac{1}{5},\,\frac{3}{5}\big\},\,i=1,2,3$,
we see that
\[
\mu\big(\mathbf{F}_{[\omega]_{m}}(\mathbb{S}_{0})\big)\le\frac{1}{2}\,\Big(\frac{5}{3}\Big)^{m}\;\Big[2\cdot\Big(\frac{3}{5}\Big)^{2m}\,\Big]=3^{m/\delta_{s}}=\nu\big(\mathbf{F}_{[\omega]_{m}}(\mathbb{S}_{0})\big)^{1/\delta_{s}}
\]
for all $\omega\in\mathrm{W}_{\ast},\,m\in\mathbb{N}$. Therefore,
\[
\mu_{n}\big(\mathbf{F}_{[\omega]_{m}}(\mathbb{S}_{0})\big)\le\nu_{n}\big(\mathbf{F}_{[\omega]_{m}}(\mathbb{S}_{0})\big)^{1/\delta_{s}},
\]
which shows that the first part of (\ref{eq:-9}) with $\overline{\delta}=\delta_{s}$.

Conversely, suppose $\mu_{n}$ satisfies the first part of (\ref{eq:-9})
for some $\overline{\delta}$. Let $S_{m}=\mathbf{F}_{1}^{m}(\mathbb{S}_{0}),\,m\in\mathbb{N}$.
Since
\[
\mu\big(\mathbf{F}_{1}^{m}(\mathbb{S}_{0})\big)=\frac{1}{2}\,\Big(\frac{5}{3}\Big)^{m}\;\Big[\Big(\frac{3}{5}\Big)^{2m}+\Big(\frac{1}{5}\Big)^{2m}\,\Big]\ge\frac{3^{m/\delta_{s}}}{2},
\]
we see that
\[
1/\overline{\delta}\le\lim_{m\to\infty}\frac{\log\mu_{n}(S_{m})}{\log\nu_{n}(S_{m})}=1/\delta_{s}.
\]
Therefore, $\overline{\delta}\ge\delta_{s}$.

(b) As in the proof of (a), we only need to prove the statements on
$\underbar{\ensuremath{\delta}}$. By Proposition \ref{prop:-3},
\[
\mu\big(\mathbf{F}_{[\omega]_{m}}(\mathbb{S}_{0})\big)\ge\frac{1}{2}\,\Big(\frac{5}{3}\Big)^{m}\;\Big(\frac{3}{25}\Big)^{m}=\frac{1}{2}\,3^{m(1+1/\delta_{s})}=\frac{1}{2}\,\nu\big(\mathbf{F}_{[\omega]_{m}}(\mathbb{S}_{0})\big)^{1+1/\delta_{s}}
\]
for all $\omega\in\mathrm{W}_{\ast},\,m\in\mathbb{N}$. Therefore,
\[
\mu_{n}\big(\mathbf{F}_{[\omega]_{m}}(\mathbb{S}_{0})\big)\ge\frac{1}{2^{n}}\,\nu_{n}\big(\mathbf{F}_{[\omega]_{m}}(\mathbb{S}_{0})\big)^{1+1/\delta_{s}}.
\]
This shows the first part of (\ref{eq:-36}) holds with $\underbar{\ensuremath{\delta}}=(1+1/\delta_{s})^{-1}$.

Conversely, suppose $\mu_{n}$ satisfies the first part of (\ref{eq:-36})
for some $\underbar{\ensuremath{\delta}}$. By Proposition \ref{prop:-3},
there exists an $\omega\in\mathrm{W}_{\ast}$ such that
\[
\lim_{m\to\infty}\big[\mathrm{trace}\big(\mathbf{Y}_{[\omega]{}_{m}}^{\mathrm{t}}\mathbf{Y}_{[\omega]_{m}}\big)\big]^{1/m}=\frac{3}{25}.
\]
Therefore,
\[
\lim_{m\to\infty}\mu\big(\mathbf{F}_{[\omega]_{m}}(\mathbb{S}_{0})\big)^{1/m}=\dfrac{5}{3}\;\lim_{m\to\infty}\big[\mathrm{trace}\big(\mathbf{Y}_{[\omega]_{m}}^{\mathrm{t}}\mathbf{Y}_{[\omega]_{m}}\big)\big]^{1/m}=\frac{1}{5},
\]
Let $S_{m}=\mathbf{F}_{[\omega]_{m}}(\mathbb{S}_{0})\times\cdots\times\mathbf{F}_{[\omega]_{m}}(\mathbb{S}_{0})$.
The above and (\ref{eq:-36}) give
\[
1/\underbar{\ensuremath{\delta}}\ge\lim_{m\to\infty}\frac{\log\mu_{n}(S_{m})}{\log\nu_{n}(S_{m})}=\frac{\log5}{\log3}=1+1/\delta_{s}.
\]

This completes the proof.
\end{proof}

\end{document}